\documentclass[11pt, letterpaper]{amsart}
\usepackage[mathscr]{eucal}
\usepackage{palatino, mathpazo, amsfonts, mathrsfs, amscd}
\usepackage{url}
\usepackage{amssymb, amsmath, amsthm}
\usepackage{soul}
\usepackage{stackrel}
\usepackage{bm}
\usepackage{stmaryrd}
\usepackage{tikz-cd}
\usepackage{tikz}
\usetikzlibrary{arrows}
\usepackage{enumitem}

\newtheorem{theorem}{Theorem}[section]
\newtheorem{lemma}[theorem]{Lemma}

\newtheorem{proposition}[theorem]{Proposition}
\newtheorem{corollary}[theorem]{Corollary}

\theoremstyle{remark}
\newtheorem{remark}[theorem]{Remark}
\newtheorem{example}[theorem]{Example}
\newtheorem{definition}[theorem]{Definition}
\newtheorem{assumption}[theorem]{Assumption}

\numberwithin{equation}{section}
\usepackage{todonotes}

\newcommand{\ol}{\overline}

\newcommand{\wt}{\widetilde}

\newcommand{\cc}[1]{\mathcal{#1}}  

\newcommand{\CC}{\mathbb{C}}
\newcommand{\ZZ}{\mathbb{Z}}

\newcommand{\QQ}{\mathbb{Q}}

\newcommand{\RR}{\mathbb{R}}

\newcommand{\h}{h}

\DeclareMathOperator{\ch}{ch}

\DeclareMathOperator{\Hom}{Hom}

\newcommand{\br}[1]{\left\langle#1\right\rangle}  

\newcommand{\op}[1]{\operatorname{#1}}

\usepackage{hyperref}

\title{A topological Chern character for matrix factorizations}

\author[Shoemaker]{Mark Shoemaker}
\address{
 \begin{tabular}{l}
	   Mark Shoemaker \\
  \hspace{.1in} Colorado State University \\
     \hspace{.1in} Department of Mathematics \\
  \hspace{.1in} 1874 Campus Delivery  \\
  \hspace{.1in} Fort Collins, CO, USA, 80523-1874\\
  \hspace{.1in} Email: {\bf mark.shoemaker@colostate.edu} \\
 \end{tabular}
}

\begin{document}

\begin{abstract}
   For $Y$ a quasi-projective complex variety and $w \colon Y \to \CC$ a regular function, we construct a Chern character from the Grothendieck group of the category of matrix factorizations of $w$ to the critical cohomology of $w$, and show that it factors through a certain topological $K$-theory group.  We prove a Grothendieck--Riemann--Roch theorem with respect to this Chern character, and verify several functorial properties.
\end{abstract}

\maketitle
\tableofcontents

For $X$ a smooth projective complex variety, the Chern character 
$$\ch\colon K^0_{\op{alg}}(X) \to H^*(X; \QQ)$$
provides a ring homomorphism from the Grothendieck group of locally free sheaves to singular cohomology.
For  $f \colon X \to X'$ a morphism of smooth projective varieties,  the Grothendieck--Riemann--Roch theorem relates the pushforward in algebraic $K$-theory to the pushforward in cohomology, giving the following commutative diagram:
 \[  \begin{tikzcd}
     K^0_{\op{alg}}(X) \ar[d, "f_*"] \ar[rr, "Td_X \cup \ch(-)"]&& H^*(X; \QQ)
     \ar[d, "f_*"] \\
            K^0_{\op{alg}}(X') \ar[rr, "Td_{X '} \cup \ch(-)"] && H^*(X' ; \QQ).
        \end{tikzcd}
        \] 

When $X$ is no longer smooth, there is not, in general, a pushforward map in either $K$-theory or cohomology.  A generalization of the Grothendieck--Riemann--Roch theorem was formulated and proven by Baum--Fulton--MacPherson in \cite{BFM},  where for $X$ a quasi-projective variety, a map
\begin{equation}\label{e:tau}
    \tau\colon K_0^{\op{alg}}(X) \to H_*(X; \QQ),
\end{equation}
is constructed from the Grothendieck group of coherent sheaves on $X$ to Borel--Moore homology.  They prove that, for $f \colon X \to X'$ a morphism of  quasi-projective varieties, 
the following diagram commutes:
 \[  \begin{tikzcd}
     K_0^{\op{alg}}(X) \ar[d, "f_*"] \ar[rr, "\tau"]&& H_*(X; \QQ)
     \ar[d, "f_*"] \\
            K_0^{\op{alg}}(X') \ar[rr, "\tau"] && H_*(X' ; \QQ).
        \end{tikzcd}
        \] 

There is a natural homomorphism
$\Phi\colon K^0_{\op{alg}}(X) \to K_0^{\op{alg}}(X).$
When $X$ is smooth, $\Phi$ is an isomorphism.
For a singular variety $X$, the quotient 
$$K_0^{\op{alg}}(X)/\op{im}(\Phi)$$
 measures the extent to which coherent sheaves on $X$ fail to be resolved by complexes of vector bundles.
 
  In this paper we construct, for a class of singular varieties $X$, a Chern character from $K_0^{\op{alg}}(X)/\op{im}(\Phi)$ to a suitable cohomology group, and prove a Grothendieck--Riemann--Roch type theorem.

\subsection{Laundau--Ginzburg models and matrix factorizations}

Let $Y$ be a smooth complex quasi-projective variety
with $w \colon Y \to \CC$ a regular function not equal to zero.  The pair $(Y, w)$ is called a \emph{Landau--Ginzburg} (LG) model.  Landau--Ginzburg models arise in many contexts, for instance as the mirrors to Fano manifolds \cite{G2, HV}, and as examples of noncommutative crepant resolutions \cite{ Sha, Asp}.

Associated to $(Y, w)$ is the derived category of matrix factorizations, $\op{MF}(Y, w)$, whose objects consist of a pair of vector bundles $E_0, E_1$, together with maps $\phi_0\colon E_0 \to E_{1}$ and $\phi_{1}: E_{1} \to E_0$ such that 
\begin{equation}
\phi_0 \circ \phi_{1} = w\cdot id_{E_0}, \hspace{1 cm}
\phi_{1}  \circ \phi_0 = w\cdot id_{E_{1}}.
\end{equation}
It was proposed by Kontsevich that the category of matrix factorizations is the correct analogue of the derived category for spaces endowed with a superpotential.
In particular, in homological mirror symmetry, $\op{MF}(Y,w)$ is the category of $D$-branes for the $B$-model of the Landau--Ginzburg model $(Y, w)$ \cite{KaLi}.

 For a quasi-projective variety $X$, the singularity category $\op{D}_{\op{Sg}}(X)$ is defined to be the Verdier quotient of $\op{D}^b(X)$ by the subcategory $\op{Perf}(X)$ of perfect complexes.   
For $(Y, w)$ a Landau--Ginzburg model, let $X=w^{-1}(0)$.  A celebrated theorem of Orlov gives an equivalence of categories 
$$\op{MF}(Y, w) \xrightarrow{\sim} \op{D}_{\op{Sg}}(X),$$
which on objects is defined by $(E_0, E_1, \phi_0, \phi_1) \mapsto \op{Coker}\phi_1.$  This implies in particular that the Grothendieck group of matrix factorizations may be identified with the quotient:
$$K_0(\op{MF}(Y, w)) = K_0^{\op{alg}}(X)/\op{im}(\Phi).$$

For $U$ a Euclidean neighborhood of $X$, let $U_{<0} = U \cap \{Re(w)^{-1}(\infty, 0)\}.$
Define the critical cohomology of $(Y, w)$ to be the direct limit
    \begin{equation}\label{e:critcohintro}
        \cc H(Y, w) := \varinjlim_U H^*(U, U_{<0})  \end{equation}
over all neighborhoods $U $ of $X$.  Alternatively, this group is equal to the cohomology of the sheaf of vanishing cycles $\phi_w(\ZZ)$.

\subsection{Results}
The main construction of this paper is a homomorphism \begin{equation}\label{e:tau1}
    \tau_{(Y, w)} \colon K_0(\op{MF}(Y, w)) \to \cc H(Y, w).
\end{equation}
With this we prove a Grothendieck--Riemann--Roch theorem for Landau--Ginzburg models.
  Suppose $f \colon Y \to Y'$ is a map of smooth quasi-projective varieties and let $w \colon Y \to \CC$ and $w' \colon Y' \to \CC$ be nonzero regular functions.  We say $f$ is a morphism of Landau--Ginzburg models if $w = w' \circ f.$  We denote it by $f \colon (Y, w) \to (Y', w').$  We prove:
\begin{theorem}[Corollary~\ref{c:GRR}]
    If $f \colon (Y, w) \to (Y', w')$ is a proper morphism of Landau--Ginzburg models, 
    then there is a commutative diagram
     \[  \begin{tikzcd}
     K_0(\op{MF}(Y, w)) \ar[d, "f_*"] \ar[r, "\tau_{(Y, w)}"]& \cc H(Y, w; \QQ) 
     \ar[d, "f_*"] \\
            K_0(\op{MF}(Y', w'))  \ar[r, "\tau_{(Y ', w ')}"] & \cc H (Y', w'; \QQ).
        \end{tikzcd}
        \] 
\end{theorem}
We further verify the compatibility of the map $\tau_{(Y, w)}$ with several other operations including pullback, shift, and tensoring with vector bundles. 
See Theorem~\ref{t:GRR} and Remark~\ref{r:cohom} for details.

We also consider the case of a map $f \colon S \to X \subset Y$ from a (possibly singular) variety $S$ to the zero locus of $w$.  In this case the Grothendieck--Riemann--Roch theorem takes the form of a commuting diagram
\begin{equation*}
\begin{tikzcd}
 K_0^{\op{alg}}(S) \ar[d, "f_*"] \ar[r, "\tau"]& H_*(S; \QQ) 
     \ar[d, "f_*"] \\
            K_0(\op{MF}(Y, w))  \ar[r, "\tau_{(Y, w)} "] & \cc H(Y, w; \QQ),
\end{tikzcd}
\end{equation*}
where $\tau$ is the homomorphism from \eqref{e:tau}.

An important property of the map $\tau$  is its compatibility with the Cartesian product of varieties under the K\"unneth map. 
The analogue of the Cartesian product for LG models is the \emph{Sebastiani--Thom sum}.   Let  $(Y_1, w_1)$ and $(Y_2, w_2)$  be two LG models. 
Let $Y = Y_1 \times Y_2$ and define the Sebastiani--Thom sum to be $$w = w_1 \boxplus w_2\colon Y \to \CC$$ where $w_1 \boxplus w_2(y_1, y_1) = w_1(y_1) + w_2(y_2).$  
There is an external product 
$$  K_0(\op{MF}(Y_1, w_1)) \otimes K_0(\op{MF}(Y_2, w_2)) \to K_0(\op{MF}(Y, w))$$
and similarly in critical cohomology.
In Theorem~\ref{t:ST} we prove that the transformation $\tau_{(Y, w)}$ also commutes with this product in an appropriate sense.

\subsection{$K$-theory}
It was proposed in  \cite{BFMKtheory} that the most natural framework for the Grothendieck--Riemann--Roch theorem is in terms of $K$-theory, as ``a transition from algebra to topology.''
They construct  a homomorphism 
$$\alpha_*\colon K_0^{\op{alg}}(X) \to K_0^{\op{top}}(X),$$
 where $K_0^{\op{top}}(X)$ is defined to be $K^0_{\op{top}}(\CC^n, \CC^n \setminus X)$ for a closed embedding $X \hookrightarrow \CC^n$.  The Riemann--Roch theorem in this context is the statement that $\alpha_*$ is covariant with respect to proper maps.  There is a Chern character map $\ch_*\colon K_0^{\op{top}}(X) \to H_*(X; \QQ),$
and the map $\tau$ of \eqref{e:tau} factors through topological $K$-theory as $\tau = \ch_* \alpha_*.$

The same turns out to be true in the context of Landau--Ginzburg models and, as in \cite{BFMKtheory}, working with $K$-theory clarifies many arguments.  In analogy with \eqref{e:critcohintro}, we define the critical (topological) $K$-theory to be the direct limit
    $$\cc K^0(Y, w) := \varinjlim_U K^0_{\op{top}}(U, U_{<0}),$$
      over all neighborhoods $U$ of $X$,
and construct a homomorphism 
$$\alpha_{(Y, w)}: K_0(\op{MF}(Y, w)) \to \cc K^0(Y, w).$$
    
    The localized Chern character of \cite{Ive, BFM} induces a map
    $$ \ch_{(Y,w)}: \cc K^0(Y, w) \to \cc H(Y, w; \QQ).$$
The map $\tau_{(Y, w)}$ is then defined to be $\tau_{(Y, w)} = Td_Y\cup \ch_{(Y, w)}(\alpha_{(Y, w)}(-)).$
In Theorem~\ref{t:GRR} we prove that $\alpha_{(Y, w)}$ is compatible with pushforward along proper maps of LG models, pullback along maps of LG models, the $K^0_{\op{alg}/\op{top}}(X)$-module structure, and shift.
The associated properties of the map $\tau_{(Y, w)}$ then follow immediately from the corresponding results for $\alpha_{(Y, w)}$.

Working in the $K$-theoretic setting also allows us to prove a compatibility between  Kn\"orrer periodicity and Bott periodicity, generalizing work of Brown \cite{Brown}, who proved a similar statement for affine hypersurface singularities.  See Corollary~\ref{c:KvB} for details.

\subsection{Relation to other works}
There has been substantial prior work constructing a Chern character for matrix factorizations in various categorical contexts.  In the equivariant affine case, a Chern character map to Hochschild homology $\op{HH}_*(\op{MF}(Y, w))$ was computed by Polishchuk--Vaintrob in \cite{PVChern} and a Hirzebruch--Riemann--Roch theorem was proved.  This was generalized to the global setting in Platt \cite{Platt}.
In \cite{Efimov}, Efimov computed the periodic cyclic homology of $\op{MF}(Y, w)$ and showed it was equal to the cohomology of the sheaf of vanishing cycles $\phi_w$.
Kim--Polishchuk constructed in \cite{KiPo} a  Chern character to $\op{HH}_*(\op{MF}(Y, w))$ by defining an analogue of the Atiyah class for matrix factorization.  This was generalized to the stacky setting and used to prove a Grothendieck--Riemann--Roch theorem by Choa--Kim--Sreedhar in \cite{CKS}.

At the $K$-theoretic level, 
in the case that $Y = \CC^n$,  Brown 
constructed a map from $K_0(\op{MF}(Y, w))$ to a version of critical topological $K$-theory in \cite{Brown}.  
In \cite{Blanc}, Blanc defined the topological $K$-theory of a general dg category over $\CC$, along with an associated topological Chern character.  
For the dg category of matrix factorizations of an LG model $(Y, w)$, this is presumably closely related to $\cc K^0(Y, w)$.  Indeed for homogeneous hypersurface singularities such a comparison is made by Brown and Dyckerhoff in \cite{BrDy}.

\subsection{Applications}

The original motivation for this paper arose from mirror symmetry.  For a large class of Landau--Ginzburg models, one can define an A-model curve counting theory known as the the gauged linear sigma model (GLSM) \cite{TX0, TX2, TX1, FJR2, CFFGKS, CKL, FKim, ShoeMirGLSM}.  In many formulations of this enumerative theory, the state space is given by the critical cohomology $\cc H(Y, w).$  In analogy with Gromov--Witten theory, one can then construct a Frobenius manifold and Dubrovin connection $\nabla^{\op{Dub}}_{(Y, w)}$ for $(Y, w)$ using generating functions of GLSM invariants as the structure constants.  
The map 
$\tau_{(Y, w)}$  may then be used to define a central charge 
$$Z^{(Y, w)}\colon K_0(\op{MF}(Y, w)) \to \CC$$
for $(Y, w)$, as well a distinguished integral lattice of  flat sections of $\nabla^{\op{Dub}}_{(Y, w)}.$

Following \cite{BMO, GGI, CIJ}, 
one expects parallel transport of flat sections of $\nabla^{\op{Dub}}_{(Y, w)}$ in the $A$-model K\"ahler moduli space to be closely related to derived equivalences of $\op{MF}(Y, w)$.  
The map 
$\tau_{(Y,w)}$ of \eqref{e:tau1} 
  provides the necessary link between $\op{MF}(Y, w)$ and $\nabla^{\op{Dub}}_{(Y, w)}$ to make such statements precise.
See \cite{ChIR, ShoeInt, ShZh} for examples of such results in particular cases.  

Using the results in this paper, we will define an $A$-model integral structure for $(Y, w)$ and investigate applications in future work.

\subsection{Statement on the use of AI}

Gemini and ChatGPT were  used to answer questions and search for references. 
All logical arguments, proofs, and writing are the author's own.

\subsection{Acknowledgements}
I would like to thank Nicolas Addington, Michael Brown, David Favero,  Tyler Kelly, Ben Knudsen,
Jeongseok Oh, and Alexander Polishchuk for valuable discussions and correspondences related to this work.  
This work was partially 
supported by the Simons Foundation Travel Grant 958189.

\section{Matrix Factorizations}

In this section we record several definitions and results about the derived category of matrix factorizations that will be used in later sections.
We mostly
follow the exposition of Orlov \cite{OrlNA}.  For the sake of simplicity, we do not work in the greatest possible generality.  For a more complete account of the theory see \cite{BFKKer, PVStacks}.

    \subsection{The derived category of matrix factorizations}
Let $Y$ be a quasi-projective complex variety and $w \colon Y \to \CC$ a regular function.  
\begin{definition} If $Y$ is smooth, we call the pair $(Y, w)$ a \emph{Landau--Ginzburg (LG) model}.
\end{definition}

\begin{definition} Given $w\colon Y \to \CC$,
a \emph{matrix factorization of $w$} is the data $\cc E =  (E_{0}, E_1, \phi_{0}, \phi_1)$ where $E_i$ are 
locally free sheaves of finite type on $Y$
and $\phi_0\colon E_0 \to E_{1}$, $\phi_{1}: E_{1} \to E_0$,  are maps satisfying:
\begin{align*}
\phi_0 \circ \phi_{1} &= w\cdot id_{E_0} \\
\phi_{1}  \circ \phi_0 &= w\cdot id_{E_{1}}.
\end{align*}

\end{definition}

There is a shift operator $[1]$ defined by 
\[\cc E[1] = (E_0, E_{-1} , -\phi_0, -\phi_1).\]
Matrix factorizations form a $\ZZ_2$-graded dg category, denoted $\op{Fact}(Y, w)$, where
morphisms 
of degree $0$ are given by 
\[\Hom_{\op{Fact}(Y, w)}^{0}(\cc E, \cc F) := \Hom_{\op{Qcoh}(Y)}(E_0, F_0 ) \oplus \Hom_{\op{Qcoh}(Y)}(E_1, F_1 ),\]
and morphisms of degree $1$ are given by 
\[\Hom_{\op{Fact}(Y, w)}^{1}(\cc E, \cc F) := \Hom_{\op{Qcoh}(Y)}(E_0, F_1 ) \oplus \Hom_{\op{Qcoh}(Y)}(E_1, F_0).\]
For a morphism $p$ of degree $k$, the differential $D$ given by
$$Dp = \phi^F \circ p - (-1)^k p \circ \phi^E.$$

Let $Z^0\op{Fact}(Y, w)$ denote the category whose objects are those of $\op{Fact}(Y, w)$ and whose morphisms are closed degree 0 morphisms of $\op{Fact}(Y, w)$.  Let $H^0\op{Fact}(Y, w)$ denote the category whose morphisms are equivalences classes of morphisms from $Z^0\op{Fact}(Y, w)$ modulo exact morphisms.

Given a morphism $p: \cc E \to \cc F \in Z^0\op{Fact}(Y, w),$ define the cone of $p$ to be the factorization 
$$\cc C(p) = \left(C(p)_0, C(p)_1, \phi^C_0, \phi^C_1\right).$$
where 
$$C(p)_0 = F_0 \oplus E_1 ,\;\;\; C(p)_1 = F_1 \oplus E_0,$$
with 
$$\phi^C_0 = \left( \begin{array}{cc}\phi^F_0 & p_1 \\
0 & -\phi^E_1
\end{array}\right),\;\;\; \phi^C_1 = \left( \begin{array}{cc}\phi_1^F & p_0 \\
0 & -\phi_0^E
\end{array}\right).$$
Define a \emph{standard triangle} in $H^0\op{Fact}(Y, w)$ to be a triangle of the form
$$ \cc E \xrightarrow{p} \cc F \xrightarrow{q} \cc C(p) \xrightarrow{r} \cc E[1]$$
where $q$ and $r$ are the obvious morphisms.  An \emph{exact triangle} is one which is isomorphic to a standard triangle.  This makes $H^0\op{Fact}(Y, w)$ into a triangulated category \cite[Proposition~2.3]{OrlNA}.

Given a complex of objects of $Z^0\op{Fact}(Y, w)$:
\[
\cc E^i \xrightarrow{d^i} \cc E^{i+1} \xrightarrow{d^{i+1}} \cdots \xrightarrow{d^{j-1}} \cc E^j,
\]
define the totalization of this complex to be the factorization $\cc T = (T_0, T_1, \phi^T_0, \phi^T_1)$ where 
$$T_0 = \bigoplus_{k+m \equiv 0 \mod 2} E_k^m ,\;\;\; T_1 = \bigoplus_{k+m \equiv 1 \mod 2} E_k^m,$$
with 
$$\phi^T_i|_{E_k^m} = d_k^m + (-1)^m \phi_k^{E^m}.$$  Let $\op{Acyc}(Y, w)$ be the thick subcategory of $H^0\op{Fact}(Y, w)$ generated by the totalizations of bounded exact complexes from $Z^0\op{Fact}(Y, w)$.  

\begin{definition}
    The \emph{derived category of matrix factorizations} of $(Y, w)$ is the triangulated category obtained as the Verdier quotient 
    $$\op{MF}(Y, w) := H^0\op{Fact}(Y, w)/ \op{Acyc}(Y, w).$$
\end{definition}

    While the derived category $\op{MF}(Y, w)$ is naturally $\ZZ_2$-graded, a slight generalization yields a $\ZZ$-graded version.  Suppose $\CC^*$ acts on $Y$ (possibly trivially) and  $w \colon Y \to \CC$ is a regular function, homogeneous of degree $d$ with respect to this action.  The weight $d$ character $\eta\colon \CC^* \to \CC^*$ defines a line bundle $\cc O_{[Y/\CC^*]}(\eta)$ on the stack $[Y/\CC^*]$, and the function
    $w$ may be viewed as a section of $\cc O_{[Y/\CC^*]}(\eta)$.  A ($\CC^*$-equivariant) factorization of $w$ is defined as before, except that  $E_0, E_1$ are now locally free sheaves of finite type on $[Y/\CC^*]$,
    $\phi_0\colon E_0 \to \cc O_{[Y/\CC^*]}(\eta) \otimes_{\cc O_{[Y/\CC^*]}} E_1  $, $\phi_1\colon E_1 \to E_0$, and 
    \begin{align*}
\phi_0 \circ \phi_{1} &= w\otimes id_{E_0} \\
\text{Id}_{\cc O_{[Y/\CC^*]}(\eta)} \otimes \phi_{1}   \circ \phi_0 &= w\otimes id_{E_{1}}.
\end{align*}
The rest of the construction follows similarly (see \cite{BFKKer} for details).
Let $\op{MF}_{\CC^*}(Y, w)$ denote the associated derived category.  
\begin{remark}
    If $\CC^*$ acts trivially on $Y$ and $w$ is the zero function, viewed as homogeneous of degree zero, there is an equivalence 
$$\op{D}^b(Y) \to \op{MF}_{\CC^*}(Y, 0)$$
that sends a vector bundle $F$ to $(F, 0, 0, 0)$ \cite[Corollary 2.3.12]{BFKvgit}.

\end{remark}

\begin{example}\label{e:geomphase}

 Let $P$ be a smooth quasi-projective variety.  Let $Y$ denote the total space of a rank $r$ vector bundle $V$ on $P$.  Let $s$ denote a regular section of $V^\vee$ and let $Z$ be the zero locus of $s$.  
     The section $s$ induces a function $w\colon Y \to \CC$.  Let $\CC^*$ act on $Y$ by scaling the fibers, and let $\eta$ denote the trivial line bundle of weight one with respect to the $\CC^*$ action.  Then $w$ can be viewed as a section of $\cc O_{[Y/\CC^*]}(\eta)$.  The $\CC^*$-equivariant LG model 
     $(Y, w)$
     is known as a \emph{geometric phase}. It may be viewed as representing the subvariety $Z$, as evidenced by the theorem below.

     Consider the diagram
\begin{equation}\label{e:csquare}\begin{tikzcd}
Y|_Z \ar[r, "j' "] \ar[d, "\pi|_Z"] & Y \ar[d, "\pi"]\\
Z  \ar[r, "j"]  & P
  \end{tikzcd}
  \end{equation}
   In work of Isik, Shipman, and Hirano, the following is proved.  \begin{theorem}[\cite{Isik, Shipman, Hirano}]\label{t:ish}
  There is an equivalence of categories given by
  \begin{equation} j'_* \circ \pi|_Z^*\colon \op{D}^b(Z) \to \op{MF}_{\CC^*}(Y, w).\end{equation}
  \end{theorem}
  
\end{example}

\subsection{Categories of singularities}

Let $X$ be a quasi-projective scheme.  Let $\op{D}^b(X)$ denote the bounded derived category of coherent sheaves on $X$.  Let $\op{Perf}(X)$ denote the subcategory of perfect complexes.  \begin{definition}\cite{Orlov}
    The \emph{category of singularities of $X$} is the triangulated category obtained as the Verdier quotient 
    $$\op{D}_{\op{Sg}}(X) = \op{D}^b(X)/\op{Perf}(X).$$
\end{definition}

Let $(Y, w)$ be an LG model with $w \neq 0$, and let $X = w^{-1}(0).$
Given a matrix factorization 
$\cc E = (E_0, E_1, \phi_0, \phi_1)$, let $E_\bullet$ denote the 2-term complex of vector bundles 
$$E_\bullet = E_1 \xrightarrow{\phi_1} E_0.$$
The cokernel $\op{Coker}\phi_1$ may be viewed as a sheaf on $X$.  A morphism $p\colon \cc E \to \cc F \in Z^0\op{Fact}(Y, w)$ induces a morphism 
$\op{Coker}\phi^E_1 \to \op{Coker}\phi^F_1$, and thus we obtain a functor
$$\op{Cok}\colon Z^0\op{Fact}(Y, w) \to \op{Coh}(X).$$  Orlov proved the following.
\begin{theorem}[\cite{Orlov, OrlNA}]\label{t:OrEq}
    The functor $\op{Cok}\colon Z^0\op{Fact}(Y, w) \to \op{Coh}(X)$ induces an equivalence of triangulated categories 
    $$\Sigma\colon \op{MF}(Y, w) \to \op{D}_{\op{Sg}}(X).$$
\end{theorem}

Denote by  
$$\Phi\colon K^0(X) \to K_0(X)$$
the  natural map from the Grothendieck group of locally free sheaves to the Grothendieck group of coherent sheaves.
\begin{corollary}\label{c:modvb}
    There are  isomorphisms of Grothendieck groups:
    $$K_0(\op{MF}(Y, w)) \xrightarrow{\cong} K_0(\op{D}_{\op{Sg}}(X)) \cong K_0(X)/\op{im}(\Phi).$$
\end{corollary}
\begin{proof}
The first isomorphism is immediate from Theorem~\ref{t:OrEq}.  For the second, we first observe that $$K_0(\op{D}_{\op{Sg}}(X)) = K_0(\op{D}^b(X))/ K_0(\op{Perf}(X)),$$ as shown in \cite[Proposition VIII.3.1.]{SGA5}   
(see also \cite[Proposition~3.5]{CTZ}).
By Orlov \cite[Remark~1.7]{Orlov} (using the fact that $X$ is quasi-projective), any perfect complex on $X$ is quasi-isomorphic to a complex of locally free sheaves of $\cc O_X$-modules. 
\end{proof}

We will use the equivalence $\Sigma$ freely in what follows to switch perspectives between matrix factorizations and objects of $\op{D}_{\op{Sg}}(X)$ when convenient.
 
\subsection{Functors}\label{ss:func}
\begin{definition}
Given $f \colon Y \to Y'$  a morphism of smooth quasi-projective varieties and regular functions $w \colon Y \to \CC$ and $w' \colon Y' \to \CC$, if $w = w' \circ f,$ we say $f$ is a \emph{morphism of Landau--Ginzburg models} and denote it by $$f\colon (Y, w) \to (Y', w').$$
\end{definition}
Given a morphism of LG models, there is a pullback functor
\[ f^* \colon \op{Fact}(Y', w') \to \op{Fact}(Y, w)\]
defined by pulling back all of the data of a factorization of $w'$.  This induces a pullback between derived categories $\op{MF}(Y', w') \to \op{MF}(Y, w),$ which we will also denote by $f^*$.

Suppose $S$ is a quasi-projective variety, and $f \colon S \to Y$ is a morphism such that $f(S) \subset X$.  Then there is a pullback
\[f^* \colon \op{MF}(Y, w) \to \op{MF}(S, 0).\]

  Let $f\colon (Y, w) \to (Y', w')$ be a map of LG models.  
Set  
  $X = w^{-1}(0)$ and $X' = {w'}^{-1}(0).$
The following is proven in \cite[Proposition~6.1 and Remark~6.2]{PVStacks}.
\begin{lemma}
    Assume that $f|_X$ is proper.  Then the derived pushforward
    $f_*\colon \op{D}^b(X) \to \op{D}^b(X')$
    sends perfect complexes to perfect complexes, and induces a pushforward 
    $$f^{\op{Sg}}_*\colon \op{D}_{\op{Sg}}(X) \to \op{D}_{\op{Sg}}(X')$$
    between singularity categories.
\end{lemma}
Define the pushforward $f_*\colon \op{MF}(Y, w) \to \op{MF}(Y', w')$ by 
$$f_* = (\Sigma ')^{-1} \circ f_*^{\op{Sg}} \circ \Sigma,$$
where
\begin{align*}
    \Sigma\colon \op{MF}(Y, w) \to \op{D}_{\op{Sg}}(X) ,\;\;\;
    \Sigma'\colon \op{MF}(Y', w') \to \op{D}_{\op{Sg}}(X')
\end{align*}
are the equivalences of Theorem~\ref{t:OrEq}.

Suppose  $v, w\colon Y \to \CC$ are two regular functions on $Y$.
 For $\cc E \in Fact (Y, v)$ and $\cc F \in Fact (Y, w)$, define $\cc E \otimes_{\cc O_Y} \cc F \in \op{Fact}(Y, v + w)$ to be the factorization  
 $$\cc E \otimes_{\cc O_Y} \cc F =\left(T_0, T_1, \phi_0^T, \phi_1^T\right).$$
where 
$$T_0 = E_1 \otimes F_1 \oplus E_0 \otimes F_0 ,\;\;\; T_1 = E_1 \otimes F_0 \oplus E_0 \otimes F_1,$$
with 
$$\phi^T_0 = \left( \begin{array}{cc}\phi^F_1 & - \phi^E_0 \\
\phi^E_1 & \phi^F_0
\end{array}\right),\;\;\; \phi^T_1 = \left( \begin{array}{cc}\phi^F_0 &  \phi^E_0 \\
-\phi^E_1 & \phi^F_1 
\end{array}\right).$$
This induces a linear map 
$$\otimes_{\cc O_Y}\colon K_0(\op{MF}(Y, v)) \otimes K_0(\op{MF}(Y, w)) \to K_0(\op{MF}(Y, v + w)). $$

There is also a tensor product 
$$\otimes_{\cc O_X}\colon  K_0(\op{Perf}(X)) \otimes K_0(\op{D}_{\op{Sg}}(X)) \to K_0(\op{D}_{\op{Sg}}(X))$$ induced by the tensor product on $\op{D}^b(X)$.  In particular,
the identification of $\op{MF}(Y, w)$ with $\op{D}_{\op{Sg}}(X)$ endows $K_0(\op{MF}(Y, w))$ with a $K^0(X)$-module structure.

One can also define an exterior tensor product.  Let  $w_1\colon Y_1 \to \CC $ and $w_2\colon Y_2 \to \CC$  be two LG models. 
Let $Y = Y_1 \times Y_2$ and define the \emph{Sebastiani--Thom sum} to be $w = w_1 \boxplus w_2\colon Y \to \CC$ where $w_1 \boxplus w_2(y_1, y_1) = w_1(y_1) + w_2(y_2).$
For $\cc E \in \op{MF}(Y_1, w_1)$ and $\cc F \in \op{MF}(Y_2, w_2)$, 
define the exterior tensor to be 
$$\cc E\boxtimes \cc F  = \pi_1^*(\cc E) \otimes_{\cc O_Y} \pi_2^*(\cc F) \in \op{MF}(Y, w).$$

\section{Critical $K$-theory}
In this section we define the cohomological and $K$-theoretic groups that will serve as targets for a Chern character for matrix factorizations.  
\subsection{Critical cohomology and $K$-theory}
Let $(Y, w)$ be an LG model as in the previous section with $w \neq 0$, and let $X = w^{-1}(0)$. 
Define the open subset
$$Y_{<0} := Re(w)^{-1}(-\infty, 0),$$
  i.e. $Y_{<0}$ is the preimage under $w$ of the half plane whose real part is negative. 
Let $Z = \op{Crit}(w) = \{dw= 0\}$.  
\begin{assumption}
Throughout the paper we always assume that $Z \subset X$.
 \end{assumption}
 
 Let $U$ be a Euclidean neighborhood of $X$ and consider the pair
$(U, U_{<0})$ where $U_{<0} = Y_{<0} \cap U$.
Given another such neighborhood $V$ of $X$ with $V \subset U$, we have a pullback $K^0_{\op{top}}(U, U_{<0}) \to K^0_{\op{top}}(V, V_{<0})$ in topological $K$-theory.  To avoid notational confusion, we will always decorate our $K$-groups with $\op{alg}$ or $\op{top}$ to distinguish between the algebraic and topological settings.
\begin{definition}\label{d:crit}
    Define the \emph{critical (topological) $K$-theory} of $(Y, w)$ to be the direct limit
    $$\cc K^0(Y, w) := \varinjlim_U K^0_{\op{top}}(U, U_{<0}),$$
    over all neighborhoods $U$ of $X$.

    Define the \emph{critical cohomology} similarly, as
    $$\cc H(Y, w) := \varinjlim_U H^*(U, U_{<0}).$$
\end{definition}
\begin{remark}
    One may also define the critical $K$-groups 
    $$\cc K^i(Y, w) := \varinjlim_U K^i_{\op{top}}(U, U_{<0}),$$
    for $i \neq 0$, although we will not use them in this paper.
\end{remark}
\begin{remark}
Although the critical cohomology can also be defined via the perverse sheaf of vanishing cycles (see e.g. \cite[Proposition~4.2.9]{Dimca} or \cite[Exercise VIII.13]{KS}), we prefer the above formulation as it has a natural analogue in $K$-theory.  Similar to the singular Riemann--Roch theorem of \cite{BFM, BFMKtheory}, most of the results and constructions below are  naturally formulated in $K$-theory.
\end{remark}

We will always take coefficients in the integers except when considering the Chern character, which will naturally land in 
$$\cc H(Y, w; \QQ) := \varinjlim_U H^*(U, U_{<0}; \QQ).$$

\begin{definition}\label{d:pair}
    Let $(A, B)$ be a pair of topological spaces.  If the pair is sufficiently nice, then $K^0(A, B)$ may be defined as the free abelian group generated by isomorphism classes of finite complexes  $$0 \to E_n \to E_{n-1} \to \cdots \to E_0 \to 0,$$
of vector bundles on $A$, exact on $B$, subject to the following three relations:
\begin{enumerate}
    \item \label{id1} if $$0 \to E_\bullet ' \to E_\bullet \to E_\bullet '' \to 0$$
    is an exact sequence of such complexes then 
    $[E_\bullet] = [E_\bullet '] + [E_\bullet ''];$
    \item \label{id2} If $E_\bullet$ is exact on $A$, then $[E_\bullet] = 0$ in $K^0(A,B);$
    \item \label{id3} If $E_\bullet$ is a complex on $A \times [0, 1]$ which is exact on $B \times [0, 1]$, then $[E_\bullet(0)] = [E_\bullet(1)]$, where $E_\bullet(t)$ denotes the restriction to $Y = Y\times \{t\}.$
\end{enumerate}
\end{definition} 
In particular, we will use the above description of relative $K$-theory for
$K^0(Y, Y\setminus X)$.
See \cite[Appendix~1]{BFMKtheory} for details.

Because $Z, X$, and $Y_{<0}$ are semialgebraic subsets of a quasi-projective variety, the triangulation theorem of Lojasiewicz \cite[Theorem~4]{Loj} implies the following.
\begin{lemma}\label{l:triangulate}
    For every neighborhood $U'$ of $Z$ (resp. $X$), there exists a neighborhood $U \subset U'$ of $Z$ (resp. $X$), a pair of finite simplicial complexes $(C_U, C_{U_{<0}})$, and a closed embedding $(C_U, C_{U_{<0}}) \hookrightarrow (U, U_{<0})$ such that $C_U \hookrightarrow U$ and $C_{U_{<0}} \hookrightarrow U_{<0} $ are deformation retracts.
\end{lemma}
\begin{proof}
    
Let $M$ be a projective variety containing $Y$ as a Zariski open subspace.  Let $Y_{\leq 0} = Y_{<0} \cup X$ and let $\overline Y_{\leq 0}$ denote the closure of $Y_{\leq 0}$ in $M$.

By \cite[Theorem~4]{Loj}, there exists a finite simplicial complex $K$ and a homeomorphism $\Phi\colon K \to M$ such that  $M \setminus Y$ is a the image of a subcomplex $A$, and such that, furthermore, there exists a subcomplex  mapping to $\overline Z$ under $\Phi$  
such that $\overline Z \setminus Z$ is the image of a subcomplex of $A$ (similarly for $\overline X$ and $\overline Y_{\leq 0}$).

Let $U'$ be a Euclidean open neighborhood of $Z$.  Because $\overline Z$ has a regular neighborhood, we may subdivide the simplices adjacent to those in $\overline Z$ in such a way that, after  replacing $K$ with this subdivision $K'$
there is a neighborhood $U \subset U'$ of $Z$ such that $U$ is the open star neighborhood of the finite  simplicial complex formed by the union of those closed simplices contained in $Z.$  

Next, perform a barycentric  subdivision of the subcomplex mapping to $\overline Y_{\leq 0}$, such that, for all simplices whose interior intersects $U_{<0}$, the added vertex lies in $U_{<0}$.  Then, let $C_U$ be the union of all closed simplices contained in $U$, and let $C_{U_{<0}} \subset C_U$ be the union of all closed simplices contained in $U_{<0}.$  Then $U$  is the open star neighborhood of $C_U$ and 
 $U_{<0}$ is the open star neighborhood of the  subcomplex $C_{U_{<0}}$.  

\end{proof} 
\begin{definition}
    
Call  an open subset $U \subset Y$ \emph{CW-finite} if there exists a pair of finite CW complexes $(C_U, C_{U_{<0}})$ and a closed embedding $(C_U, C_{U_{<0}}) \hookrightarrow (U, U_{<0})$ such that $C_U \hookrightarrow U$ and $C_{U_{<0}} \hookrightarrow U_{<0} $ are deformation retracts.

\end{definition}
If $U \subset Y$ is CW-finite, then $$K^0_{\op{top}}(U, U_{<0}) 
=   K^0_{\op{top}}(C_U, C_{U_{<0}}) 
= \wt K^0_{\op{top}}(C_U/C_{U_{<0}}).$$
By the lemma, 
$\cc K^0(Y, w)$ is equal to the limit $\varinjlim_U K^0_{\op{top}}(U, U_{<0}),$
    over all CW-finite neighborhoods $U$ of $X$.
    We use this fact for the following.
    
\begin{proposition} \label{p:ZvX}
    The critical cohomology is equal to $\cc H(Y, w) := \varinjlim_V H^*(V, V_{<0})$ over all neighborhoods $V$ of $Z.$  The same is true for the critical topological $K$-theory.
\end{proposition}

\begin{proof}
By definition, the limit
$\varinjlim_{U\supset X} H^*(U, U_{<0}; \ZZ)$ is equal to 
$$ \mathbb H^* \left(i^{-1} R\Gamma_{Y_{\geq 0}}(\ZZ) \right)$$
where $i\colon X \hookrightarrow Y$ is the inclusion, and $Y_{\geq 0} = Re(w)^{-1}([0, \infty))$.
By \cite[Proposition~4.2.9]{Dimca}, the sheaf $i^{-1} R\Gamma_{Y_{\geq 0}}(\ZZ)$ is naturally isomorphic to (a shift of) the sheaf of vanishing cycles $\phi_w(\ZZ)$. The latter is supported on the critical locus $Z$, and thus the right hand side above is equal to 
$ \mathbb H^* \left(j^{-1} i^{-1} R\Gamma_{Y_{\geq 0}}(\ZZ) \right)$ where $j\colon Z \hookrightarrow X$ is the inclusion.  This is equal to $\varinjlim_{V\supset Z} H^*(V, V_{<0}; \ZZ)$.

For the $K$-theoretic result, we apply the Atiyah-Hirzebruch spectral sequence.  For $U$ a CW-finite open subset of $Y$, we have a spectral sequence with $E_2$-page
$$E_2^{p,q} = H^p(U, U_{<0}; K^q(pt; \ZZ))$$
converging to $K^{p+q}(U, U_{<0}; \ZZ).$

Direct limits are exact in the category of abelian groups.  We therefore have two spectral sequences, obtained by taking the direct limit over all CW-finite neighborhoods of $X$ or $Z$ respectively.  There is a natural map between them, which on the $E_2$ page is given by
$$f_2\colon \varinjlim_{U\supset X} H^p(U, U_{<0}; K^q(pt; \ZZ)) \to \varinjlim_{V\supset Z}  H^p(V, V_{<0}; K^q(pt; \ZZ)).$$
By the cohomological result of the first paragraph, $f_2$ is an isomorphism.  Thus by the comparison theorem, it induces an isomorphism
$$\varinjlim_{U\supset X} K^{p+q}_{\op{top}}(U, U_{<0}; \ZZ) \to \varinjlim_{V\supset Z} K^{p+q}_{\op{top}}(V, V_{<0}; \ZZ)$$
as desired.
\end{proof}

\subsection{Homomorphisms}\label{s:homom}
Let $X \hookrightarrow Y$ be as in the previous section.
Suppose $X$ is embedded in $\CC^n$ as a closed subspace.
Define $K_0^{\op{top}}(X)$  to be $K^0_{\op{top}}(\CC^n, \CC^n \setminus X)$.
If $Y \hookrightarrow \CC^n$ is a closed embedding of $C^\infty$-manifolds and the normal bundle has a complex structure, there is a 
Thom--Gysin isomorphism \cite[Section~1.5]{BFMKtheory}: $$h_{\op{top}}\colon K^0_{\op{top}}(Y, Y\setminus X) \xrightarrow{\cong} K^0_{\op{top}}(\CC^n, \CC^n \setminus X) = K_0^{\op{top}}(X).$$

Let $f \colon (Y, w) \to (Y', w')$ be a morphism of LG models such that $w \neq 0$.  Then there is a pullback map
$$f^* \colon \cc K^0(Y', w') \to \cc K^0(Y, w).$$

If $(Y, w)$ is an LG model and $f \colon S \to Y$ is a map from a quasi-projective variety $S$ such that $f(S) \subset X$, then for any neighborhood $U$ of $X$, there is map of pairs 
$$(S, \emptyset) \to (U, U_{<0}),$$ and a pullback
$f_U^*\colon K^0_{\op{top}}(U, U_{<0}) \to K^0_{\op{top}}(S).$
This induces a pullback map
$$f^*\colon \cc K^0(Y, w) \to K^0_{\op{top}}(S).$$

Given a morphism of LG models $f \colon (Y, w) \to (Y', w'),$ let  $X' = (w')^{-1}(0)$ and   $X = w^{-1}(0) = f^{-1}(X').$
  \begin{definition}\label{d:properLG}
       We say a morphism of LG models $f \colon (Y, w) \to (Y', w')$ is \emph{proper} if there is a Euclidean neighborhood $U'$ of $X'$ such that 
       $$f|_{f^{-1}(U')} \colon f^{-1}(U') \to U'$$
       is proper.
  \end{definition}
 
Proper pushforward in $K$-theory induces a pushforward map in critical $K$-theory for proper morphisms of LG models.
\begin{proposition}
    Let $f \colon (Y, w) \to (Y', w')$ be a proper morphism of LG models.   There is a pushforward map 
    $f_*\colon \cc K^0(Y, w) \to \cc K^0(Y', w')$ (and similarly for critical cohomology).  
    Furthermore, the following diagram commutes
      \begin{equation}  \label{e:PushZCrit}
      \begin{tikzcd}
     K_0^{\op{top}}(X) \ar[d, "f_*"] \ar[r, "i_{<0}^* h_{\op{top}}^{-1}"]& \cc K^0(Y, w)
     \ar[d, "f_*"] \\
           K_0^{\op{top}}(X ')  \ar[r, "(i ')_{<0}^*h_{\op{top}}^{-1}"] & \cc K^0(Y' , w').
        \end{tikzcd}
        \end{equation}
\end{proposition}

\begin{proof}
    We will apply  the bivariant theory of \cite{FuMa}.  Suppose that $A$ and $B$ are topological spaces, each embeddable as closed subsets of Euclidean space $\RR^n.$
In  \cite[Section~3]{FuMa} for a map $f \colon A \to B$ between such spaces, one may define a graded abelian group $K^0(A \to B).$  If $\phi \colon A \to \RR^n$ is a closed  embedding then 
$$K^0(A \to B) := K^0_{\op{top}}( B \times \RR^n, B \times \RR^n \setminus (f,\phi)(A)).$$  In particular, if $f$ is a closed embedding, then $K^0(A \to B) = \\ K^0_{\op{top}}(B, B\setminus A).$   The properties satisfied by these bivariant groups are given in \cite[Section~1]{FuMa}, where in this case the confined morphisms are proper maps, and independent square are fiber squares. In particular, given $$A \xrightarrow{f} B \xrightarrow{g} C$$ with $f$ proper, there is a bivariant pushforward 
$$f_*^{\op{biv}}: K^0(A \to C) \to K^0(B \to C).$$ 

Given an LG model $(Y, w)$, we have, by definition, an identification $K^0_{\op{top}}(Y, Y_{<0}) = K^0(Y_{\geq 0} \to Y),$
where $Y_{\geq 0} = Re(w)^{-1}([0,\infty)).$ 

   Suppose $f\colon (Y, w) \to (Y', w')$ is a proper morphism of LG models.
     If the map on spaces $f\colon Y \to Y'$ is proper one may define a pushforward map $$f_*\colon K^0(Y_{\geq 0} \to Y) \to  K^0(Y'_{\geq 0} \to Y').$$   To be more precise, because $Y$ and $Y'$ are smooth, there is a canonical $K$-theoretic orientation class $\{f\} \in K^0(Y \to Y')$ and the map $f_*$ above is the composition of the cup product  
     $$-\cup \{f\}\colon K^0(Y_{\geq 0} \to Y) \to K^0(Y_{\geq 0} \to Y')$$ 
     with the bivariant pushforward $(f|_{Y_{\geq0}})_*^{\op{biv}}$ associated to the composition
     $$Y_{\geq 0} \xrightarrow{f|_{Y_{\geq0}}} Y'_{\geq 0} \to Y'.$$
    By assumption, there exists a neighborhood  $U'$ of $X'$ such that $f|_{f^{-1}(U')}$
    is proper.  Then by the same argument as in the previous paragraph there exists a pushforward 
\begin{equation}\label{e:pushK}
        K^0(f^{-1}(U')_{\geq 0} \to f^{-1}(U')) \to K^0(U'_{\geq 0} \to U').
    \end{equation}
    To simplify notation, let us replace $Y$ by $f^{-1}(U')$ if necessary, and assume for the remainder of the proof that $f \colon Y \to Y'$ is a proper map of topological spaces.

    To define the pushforward in critical $K$-theory, we must check that,
    for any neighborhood $U$ of $X$, there exists a neighborhood $U'$ of $X'$ such that $f^{-1}(U')$ is contained in $U$.  
    
    Given $U$ a neighborhood of  $X$, note first that $f^{-1}(X') \subset U$.
    Take $x' \in X'$, let $Y_{x'}$ denote the fiber $f^{-1}(x')$.  Let $B_k(x')$ be a sequence of open neighborhoods of $x'$ such that 
    $$\cap_k B_k(x') = x'.$$ Assume without loss of generality that $B_k(x') \subset B_1(x')$ for all $k$, and that $\ol{B_1(x')}$ is compact.
    Consider $V_k(x') = f^{-1}(B_k(x'))$.  Then for some $k$, we claim that $V_k(x') \subset U$.  To see this, assume the contrary, then for each $k$ there exists a $y_k \in V_k(x') \setminus U$.  Then 
    $\{y_k\}_{k=1}^\infty$ is a sequence contained in $f^{-1}(\ol{B_1(x')})$, which is compact by properness of $f$.  Therefore there exists a convergent subsequence $\{y_{k_i}\}_{i=1}^\infty$.  Because $U$ is open, the limit $y_{k_\infty}$ must lie in the boundary of $U$, and therefore is not contained in  $Y_{x'}$.  This is a contradiction, as 
    $$f(y_{k_\infty}) = \lim_i f(y_{k_i}) = x',$$
    where the last equality is because $f(y_{k_i}) \in B_{k_i}(x').$  This proves the claim.  Next, choose a $k$ such that $V_k(x') \subset U$ and let $B(x') = B_k(x').$  Then  
    $$U' = \cup_{x' \in X'} B(x')$$
    gives the desired open subset.  
    
    The map $\cc K^0(Y, w) \to \cc K^0(Y', w')$ may then be defined using the composition of the restriction
    $K^0(U_{\geq 0} \to U) \to K^0(f^{-1}(U')_{\geq 0} \to f^{-1}(U'))$ with the map \eqref{e:pushK}.  It is left to check that this induces a well-defined map $\cc K^0(Y, w) \to \cc K^0(Y', w')$.  For $X' \subset U'' \subset U'$, consider the following fiber diagram:
    \begin{equation} \label{e:pushpull diagram}
    \begin{tikzcd}
    f^{-1}(U'')_{\geq 0} \ar[d, "f '' "] \ar[r] & f^{-1}(U')_{\geq 0} \ar[d, "f' "] \\
    U''_{\geq 0} \ar[d] \ar[r] & U'_{\geq 0} \ar[d] \\
    U'' \ar[r, "i ''"] & U'   
    \end{tikzcd}
    \end{equation}
    By \cite[Axiom (A$_{23}$) of Section~2.2]{FuMa}, $(i'')^* (f')^{\op{biv}}_* = (f'')_*^{\op{biv}} (i'')^*.$ This gives the requisite compatibility.  
    
    For the second statement, consider $U'$ a neighborhood of $X'$.  We have the diagram
\[  \begin{tikzcd}
     K_0^{\op{top}}(X) \ar[d, "f_*"] \ar[rr, "i_{(f^{-1}(U'), f^{-1}(U')_{<0})}^*"]&& K_0^{\op{top}}(f^{-1}(U')_{\geq 0})
     \ar[d, "f_*"] \\
            K_0^{\op{top}}(X')  \ar[rr, "i_{(U', U'_{<0})}^*"] && K_0^{\op{top}}(U'_{\geq 0}),
        \end{tikzcd}
        \] 
     which commutes by functoriality of proper pushforward \cite[Axioms (A$_2$) and (A$_{12}$)]{BFMKtheory}.  This implies \eqref{e:PushZCrit}.
     
\end{proof}

We conclude by remarking that $\cc K^0(Y, w)$ has a natural $K^0_{\op{top}}(X)$-module structure, by identifying $K^0_{\op{top}}(X)$ with $K^0_{\op{top}}(U)$ for $U$ a neighborhood of $X$ which is homotopy equivalent to $X$ and observing that $\cc K^0(Y, w)$ is naturally a $K^0_{\op{top}}(U)$-module.

\subsection{Geometric phases} \label{ss:geomphase}
     Recall the geometric phase LG models defined in Example~\ref{e:geomphase}.  Following the notation from that section, we have the following analogue of Theorem~\ref{t:ish}.
\begin{proposition}\label{p:BPK}
    There is a canonical isomorphism 
    $$j'_* \circ \pi|_Z^*\colon K_i^{\op{top}}(Z) = \cc K^i(Y, w).$$
\end{proposition}

\begin{proof}
The zero locus $X$ is $P \cup Y|_Z.$ By \cite[Theorem~4]{Loj} every neighborhood of $Z$ in $P$ contains a smaller neighborhood $N$ of $Z$ which deformation retracts onto $Z$.  If $T$ is a tubular neighborhood of $P$ in $Y$, then $T \cup Y|_N$ is a neighborhood of $X$.  Furthermore, for any neighborhood  $U$ of $X$, one can find a neighborhood $W$ of $X$ equal to $T \cup Y|_N,$ for some $N$ as above.  Note that $(T \cup Y|_N, (T \cup Y|_N)_{<0})$ is homotopy equivalent to  $(T, T_{<0})$.
Therefore we have canonical isomorphisms
$$K^i_{\op{top}}(W, W_{<0}) = K^i_{\op{top}}(T, T_{<0}) = K^i_{\op{top}}(Y, Y_{<0}).$$  Thus for any neighborhood $U$ of $X$ we have pullbacks
$$K^i_{\op{top}}(Y, Y_{<0}) \to K^i_{\op{top}}(U, U_{<0}) \to K^i_{\op{top}}(W, W_{<0}) = K^i_{\op{top}}(Y, Y_{<0})$$
such that the composition is an isomorphism.  This induces a map 
$$\cc K^i(Y, w) \to K^i_{\op{top}}(Y, Y_{<0})$$
which is inverse to the pullback $K^i_{\op{top}}(Y, Y_{<0}) \to \cc K^i(Y, w).$
   
By scaling the fiber coordinate, one can deformation retract $Y_{<0}$ to $F = w^{-1}(-1).$
  The five lemma then implies that  $ K^i_{\op{top}}(Y, Y_{<0}) = K^i_{\op{top}}(Y, F).$  So we can replace $\cc K^i(Y, w)$ with $K^i_{\op{top}}(Y, F)$ for the rest of the proof.

  The projection $F \to P \setminus Z$ is a locally trivial fiber bundle with fiber isomorphic to $\CC^{r-1}$.  Thus the projection $F \to P \setminus Z$ is a homotopy equivalence.  

  Then we have the following map between long exact sequences.
\[
\begin{tikzcd}[column sep=small]
 K^{i-1}_{\op{top}}(P) \ar[r] \ar[d, "="] &K^{i-1}_{\op{top}}(P \setminus Z) \ar[r] \ar[d] & K^i_{\op{top}}(P, P \setminus Z) \ar[r] \ar[d]   & K^i_{\op{top}}(P) \ar[r] \ar[d, "="] &K^i_{\op{top}}(P \setminus Z) \ar[d]  \\
 K^{i-1}_{\op{top}}(Y) \ar[r] \ar[d, "="] &K^{i-1}_{\op{top}}(Y \setminus Y|_Z) \ar[r] \ar[d] & K^i_{\op{top}}(Y, Y \setminus Y|_Z) \ar[r] \ar[d]   & K^i_{\op{top}}(Y) \ar[r] \ar[d, "="] &K^i_{\op{top}}(Y \setminus Y|_Z) \ar[d] \\
K^{i-1}_{\op{top}}(Y) \ar[r] & K^{i-1}_{\op{top}}(F)\ar[r] & K^i_{\op{top}}(Y, F) \ar[r] & K^i_{\op{top}}(Y)\ar[r] & K^i_{\op{top}}(F) 
\end{tikzcd}
\]
The four outer vertical compositions are isomorphisms, thus the center must be as well.
Via the Thom--Gysin isomorphisms $K_i^{\op{top}}(Z) \to K^i_{\op{top}}(P, P \setminus Z)$ and  $K_i^{\op{top}}(Y|_Z) \to K^i_{\op{top}}(Y, Y \setminus Y|_Z)$, the top middle map is identified with $\pi|_Z^*,$  and lower middle map is $j'_*.$
\end{proof}

\section{Chern character}\label{s:checha}

In this section we give the main construction of the paper, a homomorphism $\alpha_{(Y,w)}$
from the Grothendieck group of the category of matrix factorizations of $w$ to the critical topological $K$-theory.

\subsection{Riemann--Roch and topological $K$-theory}
We recall the Riemann--Roch theorem of \cite{BFMKtheory}.  Let $X$ be a quasi-projective variety, and $X \hookrightarrow Y$ an closed embedding into a smooth quasi-projective variety.

Let $K^0_{\op{alg}}(Y, Y\setminus X)$ denote the Grothendieck group of finite complexes of locally free sheaves on $Y$ which are exact off of $X$, subject to the relations \eqref{id1} and \eqref{id2} of Definition~\ref{d:pair}.  There is a homomorphism $$\alpha^* \colon K^0_{\op{alg}}(Y, Y\setminus X) \to K^0_{\op{top}}(Y, Y\setminus X)$$ mapping a complex of locally free sheaves on $Y$ to the associated complex of topological vector bundles.

Let $K_0^{\op{alg}}(X) = K_0\left( \op{D}^b(X)\right)$ denote the Grothendieck group of coherent sheaves on $X$ 
and recall $K_0^{\op{top}}(X)=K^0_{\op{top}}(\CC^n, \CC^n \setminus X)$ was defined in Section~\ref{s:homom}.
There are isomorphisms 
\begin{equation}\label{e:halgtop}
h_{\op{alg}/\op{top}}\colon K^0_{\op{alg}/\op{top}}(Y, Y\setminus X) \xrightarrow{\cong} K_0^{\op{alg}/\op{top}}(X).
\end{equation}  
In the algebraic case, $h_{\op{alg}}(F_\bullet) = \sum (-1)^i H^i(F_\bullet),$ while in the topological case $h_{\op{top}}$ is the Thom--Gysin map \cite[Section~1.5]{BFMKtheory}: $$K^0_{\op{top}}(Y, Y\setminus X) \xrightarrow{\cong} K^0_{\op{top}}(\CC^n, \CC^n \setminus X).$$
Define the map
$\alpha_*\colon K_0^{\op{alg}}(X) \to K_0^{\op{top}}(X),$
as the composition $$\alpha_* := h_{\op{top}} \circ \alpha^* \circ h_{\op{alg}}^{-1}.$$

The Grothendieck--Riemann--Roch theorem of \cite{BFMKtheory} 
is the statement that for $f \colon X \to X'$ a proper map of algebraic varieties, 
 the diagram 
     \begin{equation}\label{e:GRRKSing}  \begin{tikzcd}
     K_0(\op{D}^b(X)) \ar[d, "f_*"] \ar[r, "\alpha_*"]& K_0^{\op{top}}(X) 
     \ar[d, "f_*"] \\
            K_0(\op{D}^b(X'))  \ar[r, "\alpha_* '"] & K_0^{\op{top}}(X')
        \end{tikzcd}
       \end{equation}
commutes.
 
As in the algebraic setting, there is a cap product $$K^0_{\op{top}}(X) \otimes K_0^{\op{top}}(X) \to K_0^{\op{top}}(X)$$ which may be defined as follows.  
Let $U$ be an open neighborhood of $X$ in the Euclidean topology such that the inclusion $X \to U$ is a homotopy equivalence.
Then $K^0_{\op{top}}(X) = K^0_{\op{top}}(U)$, $K_0^{\op{top}}(X)$ is identified with $ K^0_{\op{top}}(U, U \setminus X)$ via $h_{\op{top}}$, and the cap product is given by tensor product. 
The map $\alpha_*$ is compatible with the $K^0_{\op{alg}/\op{top}}(X)$-module structure \cite[Section~4.1]{BFMKtheory}.  Namely,
 for $F \in K^0_{\op{alg}}(X)$ and $S \in K_0^{\op{alg}}(X)$,
\begin{equation}
    \label{e:module}
    \alpha_*(F \otimes S) = \alpha^*(F) \otimes \alpha_*(S).
\end{equation}

\subsection{The construction of $\alpha_{(Y, w)}$}

Suppose $(Y, w)$ is an LG model and $w \neq 0$.  Let $\cc E = (E_0, E_1, \phi_0, \phi_1) \in \op{Ob}\left( \op{MF}(Y, w)\right)$ be a matrix factorization of $w$.  Because $d^2 = w \cdot \op{Id}$, both $\phi_0$ and $\phi_1$ are isomorphisms away from $X = w^{-1}(0).$  
Thus the  complex of vector bundles
\begin{equation}\label{e:assoc}
    E_\bullet:= E_1 \xrightarrow{\phi_1} E_0\end{equation}
is exact on $Y \setminus X.$  Consider the complex of topological vector bundles  $\alpha^*(E_\bullet) \in K^0_{\op{top}}(Y, Y\setminus X)$.

Given an open neighborhood $U$ of $X$, consider the map of pairs 
$$i_{(U, U_{<0})}\colon (U, U_{<0}) \to (Y, Y \setminus X).$$  
Taking the limit over all such open neighborhoods, the pullbacks $$ i_{(U, U_{<0})}^*\colon K^0_{\op{top}}(Y, Y\setminus X) \to K^0_{\op{top}}(U, U_{<0})$$  induce a map  \begin{equation}\label{e:pullbacklim}
    i^*_{<0}\colon K^0_{\op{top}}(Y, Y\setminus X) \to \cc K^0(Y, w). 
\end{equation}
 For a matrix factorization $\cc E$, consider the associated class in critical $K$-theory
\begin{equation}
\label{e:ayw1}  i_{<0}^*\left(\alpha^*(E_\bullet)\right) \in \cc K(Y, w).
\end{equation}

The main result of this section is that \eqref{e:ayw1} depends only on the class of $\cc E$ in the Grothendieck group of $\op{MF}(Y, w).$
\begin{theorem}\label{t:MFch} 
    The map $ \op{Ob}\left( \op{MF}(Y, w)\right) \to \cc K^0(Y, w)$ given by 
    $$\cc E \mapsto i_{<0}^*\left(\alpha^*(E_\bullet)\right)$$
    induces a homomorphism:
    $$\alpha_{(Y,w)}\colon K_0(\op{MF}(Y, w)) \to \cc K^0(Y, w).$$
\end{theorem}

Before proving the theorem, we give an equivalent statement.   For a matrix factorization $\cc E = (E_0, E_1, \phi_0, \phi_1)$,  $\phi_1$ is injective and so
$h_{\op{alg}}(E_\bullet) = \op{Coker}(\phi_1)$.  We therefore have the equality
\begin{equation}
    i_{<0}^*\left(\alpha^*(E_\bullet)\right)= i_{<0}^* \left( h_{\op{top}}^{-1}\alpha_*(\op{Coker}(\phi_1))\right),
\end{equation}
where $\h_{\op{top}}$ is the map \eqref{e:halgtop}.

By Corollary~\ref{c:modvb}, Theorem~\ref{t:MFch} is equivalent to the following statement, which is what we will prove.
\begin{theorem}
    The map 
    $i_{<0}^* h_{\op{top}}^{-1}\alpha_*\colon K_0(\op{D}^b(X)) \to \cc K^0(Y, w)$
    induces a map
    $$\alpha_{\op{Sg}}\colon K_0(\op{D}_{\op{Sg}}(X)) \to \cc K^0(Y, w).$$
\end{theorem}

\begin{proof}

By Corollary~\ref{c:modvb}, it is enough to show that for a locally free sheaf $P$, $i_{(U, U_{<0})}^* \left(h_{\op{top}}^{-1}\alpha_*(P)\right) = 0$ for a sufficiently small neighborhood $U$ of $X$.

    To begin, consider the special case that $P = \cc O_X$.  This is resolved by locally free sheaves on $Y$ by 
$$\cc O_Y \xrightarrow{\cdot w} \cc O_Y.$$
Therefore $\alpha_*(\cc O_X) = h_{\op{top}}(\alpha^*(\cc O_Y \xrightarrow{\cdot w} \cc O_Y)).$
But  $\cc O_Y \xrightarrow{\cdot w} \cc O_Y$ is the pullback of $\cc O_{\CC} \xrightarrow{x} \cc O_{\CC}$ via the map $w \colon Y \to \CC$ (where $x$ is the coordinate function on $\CC$).    There is a commuting diagram of pairs
    \[ \begin{tikzcd}
        (Y, Y_{<0}) \ar[d, "w"] \ar[r, "i_{(Y, Y_{<0})}"] & (Y, Y \setminus X) \ar[d, "w"] \\
        (\CC, \{Re(z) < 0\}) \ar[r, "i_1"] & (\CC, \CC \setminus 0).
    \end{tikzcd} \]
It follows that 
$$i_{(Y, Y_{<0})}^* \alpha^*(\cc O_Y \xrightarrow{\cdot w} \cc O_Y) = i_{(Y, Y_{<0})}^* w^* \left(\alpha^*(\cc O_{\CC} \xrightarrow{x} \cc O_{\CC}) \right) = w^* i_1^* \left( \alpha^*(\cc O_{\CC} \xrightarrow{x} \cc O_{\CC})\right),$$
but the right hand side must be zero, because $K^0(\CC, \{Re(z) < 0\}) = 0.$

Next, suppose $P$ is a locally free sheaf on $X$.  Let $U$ be an open neighborhood of $X$ in the Euclidean topology such that the inclusion $X \to U$ is a homotopy equivalence, and denote by $\wt P$ the topological vector bundle on $U$ which restricts to $\alpha^*(P)$ when pulled back to $X$.
By \eqref{e:module}, the following diagram commutes:
\[\begin{tikzcd}
  K^0_{\op{alg}}(X)\otimes K_0^{\op{alg}}(X) \ar[r] \ar[d, "\alpha^* \otimes \alpha_*"] & K_0^{\op{alg}}(X) \ar[d, "\alpha_*"] \\
  K^0_{\op{top}}(X)\otimes K_0^{\op{top}}(X) \ar[r] \ar[d, "\cong"] & K_0^{\op{top}}(X) \ar[d, "\cong"] \\
  K^0_{\op{top}}(U)\otimes K^0_{\op{top}}(U, U \setminus X) \ar[r] \ar[d] & K^0_{\op{top}}(U, U \setminus X) \ar[d, , "i_{(U, U_{<0})}^*"] \\
  K^0_{\op{top}}(U)\otimes K^0_{\op{top}}(U, U_{<0}) \ar[r] & K^0(U, U_{<0}) \ar[u, phantom]
\end{tikzcd}
\]
We have
\begin{align*}
    \alpha_*(P) 
    & =  \alpha_*(P \otimes \cc O_X) \\
    & = \alpha^*(P) \otimes \alpha_*(\cc O_X) \\
    &=  \wt P \otimes h_{\op{top}}\left(\alpha^*(\cc O_Y \xrightarrow{\cdot w} \cc O_Y )\right)  \\
    & = h_{\op{top}}\left(\wt P \otimes \alpha^*(\cc O_Y \xrightarrow{\cdot w} \cc O_Y )\right).
\end{align*}
We conclude that
\begin{align*} 
i_{(U, U_{<0})}^*(h_{\op{top}}^{-1}(\alpha_*(P))) &=
i_{(U, U_{<0})}^*\left(\wt P \otimes \alpha^*(\cc O_Y \xrightarrow{\cdot w} \cc O_Y )\right) \\
&= i_{(U, U_{<0})}^*\left(\wt P\right) \otimes i_{(U, U_{<0})}^*\left(\alpha^*(\cc O_Y \xrightarrow{\cdot w} \cc O_Y )\right) \\
&= i_{(U, U_{<0})}^*\left(\wt P\right) \otimes 0,
\end{align*}
where the last equality is the first part of the proof.

\end{proof}

\subsection{Chern character}\label{ss:cc}

For $Z$ a closed subset of a $C^\infty$-manifold $Y$, recall the localized Chern character $\ch_{\op{loc}}\colon K^0_{\op{top}}(Y, Y \setminus Z) \to H^*(Y, Y \setminus Z; \QQ)$ constructed by Iversen and Baum--Fulton--MacPherson \cite{Ive, BFM}. Following \cite{BFM}, define $$ch_*\colon K_0^{\op{top}}(X) = K^0_{\op{top}}(\CC^n, \CC^n \setminus X) \xrightarrow{ch_{\op{loc}}} H^*(\CC^n, \CC^n \setminus X; \QQ) \xrightarrow{[\CC^n] \cap -} H_*(X; \QQ).$$ 
Alternatively,
for $X$ a closed subset of a smooth complex variety $Y$, $\ch_*$ is   equal to the composition
$$K_0^{\op{top}}(X) \xrightarrow{h_{\op{top}}^{-1}} K^0_{\op{top}}(Y, Y \setminus X) \xrightarrow{\op{Td}_Y \cup \ch_{\op{loc}}}  H^*(Y, Y \setminus X; \QQ) \xrightarrow{[Y] \cap -} H_*(X; \QQ).$$ 
The map $\ch_*$ is covariant for proper maps \cite[Section~5]{BFMKtheory}.  

  Given an inclusion of neighborhoods $i'\colon U' \hookrightarrow U$, the localized Chern character  $$\ch_{\op{loc}}\colon K^0(U, U_{<0}) \to H^*(U, U_{<0}; \QQ)$$ 
  commutes with pullback along $i'$, thus  
there is an  induced a homomorphism on critical $K$-theory:
\begin{equation}\label{e:chyw}
    \ch_{(Y,w)}\colon \cc K^0(Y, w) \to \cc H(Y, w; \QQ).
\end{equation}

\begin{definition}
Define 
$\tau_{(Y, w)}\colon K_0(\op{MF}(Y, w)) \to \cc H(Y, w; \QQ)$ by
$$\tau_{(Y, w)}(\cc E) = \op{Td}_Y\cup \ch_{(Y, w)}(\alpha_{(Y, w)}(\cc E)).$$
\end{definition}

    If we identify $H_*(X; \QQ)$ with $H^*(Y, Y \setminus X; \QQ)$ via the cap product with $[Y]$, then $\tau_{(Y, w)}$ is induced by the composition $i_{<0}^* \circ \tau$, where 
    $\tau\colon K_0^{\op{alg}}(X) \to H_*(X;\QQ) = H^*(Y, Y \setminus X; \QQ)$ is the map in the singular Riemann--Roch theorem of \cite{BFM} and 
    $i^*_{<0}\colon H^*(Y, Y\setminus X) \to \cc H(Y, w)$ is the cohomological analogue of \eqref{e:pullbacklim}.  In other words, the following diagram commutes by construction:
  \[	\begin{tikzcd}
\ar[d, "\tau"] \ar[r] K_0^{\op{alg}}(X) & K_0(\op{D}_{\op{Sg}}(X)) \ar[r, "\Sigma^{-1}"]& K_0(\op{MF}(Y, w)) \ar[d, "\tau_{(Y, w)}"] \\
H_*(X;\QQ) \ar[r] & H^*(Y, Y \setminus X; \QQ) \ar[r, "i_{<0}^*"] &
 \cc K^0(Y, w)
	\end{tikzcd}
\]
where the bottom left arrow is the inverse of the map $[Y] \cap -: H^*(Y, Y \setminus X; \QQ) \to H_*(X;\QQ) $.

\subsection{2-periodic complexes}
We conclude this section with a verification that the natural map 
$$\alpha^*: K^0_{\op{alg}}(X) \to K^0_{\op{top}}(X)$$
sending a locally free sheaf to the associated topological vector bundle has a 2-periodic analog.  

Let $S$ be a quasi-projective variety, possibly singular.  Recall that $\op{MF}(S, 0)$ denotes the derived category of matrix factorizations of zero (i.e. 2-periodic complexes) on $S$.
\begin{proposition}
    There is a map $$\alpha^*_{\ZZ/2} \colon K_0(\op{MF}(S, 0)) \to K^0_{\op{top}}(S) $$ given by 
    $$[\cc E] \mapsto [E_0] - [E_1].$$
\end{proposition}
\begin{proof}
We first claim that if $\cc E = (E_0, E_1, \phi_0, \phi_1),$ is null-homotopic, then $ E_0 \cong E_1$.
By assumption we have maps $h_0\colon E_0 \to E_1$ and $h_1\colon E_1 \to E_0$ such that 
\begin{equation}\label{e:homotopy}
    id_0 = h_1 \phi_0 + \phi_1 h_0, \hspace{1 cm} id_1 = h_0 \phi_1 + \phi_0 h_1.    
\end{equation}
Define the endomorphisms 
$e_A= \phi_1 h_0, e_B = h_1 \phi_0\colon E_0 \to E_0$ and 
$e_C = \phi_0 h_1, e_D = h_0 \phi_1\colon E_1 \to E_1.$  It follows from \eqref{e:homotopy} that each of the above is idempotent. Define the locally free sheaves $A = \op{im}(e_A)$, $B = \op{im}(e_B)$, $C = \op{im}(e_C)$, $D = \op{im}(e_D).$ One checks that 
$$E_0 = A \oplus B, \hspace{1 cm} E_1 = C \oplus D,$$
and $\phi_0|_B\colon B \to C$ and $h_0|_A\colon A \to D$ are isomorphisms (with inverses $h_1|_C$ and $\phi_1|_D$ respectively).  The claim follows.

A map $p\colon \cc E \to \cc F$ is an isomorphism in $H^0\op{Fact}(S, 0)$ if it has an inverse up to homotopy, or equivalently, if $C(p)$ is null-homotopic.  By the previous paragraph, in this case $C(p)_0 = F_0 \oplus E_1$ is isomorphic to $C(p)_1 = F_1 \oplus E_0.$ Therefore if $\cc E$ and $\cc F$ are isomorphic in $H^0\op{Fact}(S, 0),$ then
$[E_0] - [E_1] = [F_0] - [F_1]$
in $K^0_{\op{top}}(S).$

For a standard triangle
$$\cc E \xrightarrow{p} \cc F \to C(p) \to \cc E[1]$$
in $H^0\op{Fact}(S, 0),$ it is immediate that
$$[F_0] - [F_1] = [E_0] - [E_1] + [C(p)_0] - [C(p)_1].$$
We conclude that $\cc E \to [E_0] - [E_1]$ induces a well-defined map $$\nu \colon K_0(H^0\op{Fact}(S, 0)) \to K^0_{\op{top}}(S).$$

Finally, one must verify that the objects of $\op{Acyc}(S, 0)$ map to zero under $ \nu$, thus inducing the desired map $\alpha^*_{\ZZ/2}$.  An acyclic factorization is locally contractible, hence the associated 2-periodic complex of topological vector bundles is exact on $Y$.

Let $(E_0, E_1, \phi_0, \phi_1)$ be an exact 2-periodic complex of topological vector bundles.  Choose a Hermitian metric on $E_0$ and $E_1.$  Define the map of vector bundles
$$D = \phi_1 + \phi_0^*: E_1 \to E_0$$
where  $\phi_0^*$ is the adjoint of $\phi_0$ with respect to the chosen metrics.  Consider the map $D^* D: E_1 \to E_1$
\begin{align*}D^* D &= (\phi_1^* + \phi_0)(\phi_1 + \phi_0^*) \\
& =\phi_1^* \phi_1 + \phi_0 \phi_0^*
\end{align*}
where the last equality is because $\phi_0 \phi_1 = \phi_1^* \phi_0^* = 0.$
Then for $\vec v \in E_1|_s$, 
\begin{equation} \label{e:Hermes}
    \br{D^* D \vec v, \vec v} = \br{\phi_1 v, \phi_1 v} + \br{\phi_0^* \vec v, \phi_0^* \vec v} 
\end{equation} 
If $\vec v \in \ker(\phi_1)$, then by exactness $\vec v \in \op{im}(\phi_0)$, and therefore does not lie in $\ker(\phi_0^*)$.  Thus \eqref{e:Hermes} is nonzero for all $\vec v \neq 0$.  In particular $D^* D$ is an isomorphism, so $D$ is an isomorphism.
\end{proof} 

\section{Grothendieck--Riemann--Roch}

 In this section we verify several functorial properties of the map $\alpha_{(Y, w)}$, analogous to those appearing in 
  \cite{BFMKtheory}.

\subsection{Grothendieck--Riemann--Roch in $K$-theory}

\begin{theorem} \label{t:GRR}
Let $f \colon Y \to Y'$ be a morphism of LG models.  Assume that $w$ is not zero  on $Y$.  Then,
    \begin{enumerate}
        \item \label{i:push} if $f$ is proper (Definition~\ref{d:properLG}), there is a commutative diagram
     \[  \begin{tikzcd}
     K_0(\op{MF}(Y, w)) \ar[d, "f_*"] \ar[r, "\alpha_{(Y, w)}"]& \cc K^0(Y, w) 
     \ar[d, "f_*"] \\
            K_0(\op{MF}(Y', w'))  \ar[r, "\alpha_{(Y ', w ')}"] & \cc K^0(Y', w');
        \end{tikzcd}
        \] 
        \item \label{i:pull} there is a commutative diagram
        \[  \begin{tikzcd}
         K_0(\op{MF}(Y', w')) \ar[d, "f^*"] \ar[r, "\alpha_{(Y ', w ')}"] &  \cc K^0(Y', w') \ar[d, "f^*"] \\
     K_0(\op{MF}(Y, w))  \ar[r, "\alpha_{(Y, w)}"]& \cc K^0(Y, w).
        \end{tikzcd}
        \] 
\end{enumerate}
Let $f\colon S \to Y$ be a morphism from a complex variety $S$ to $Y$ such that $f(S) \subset X.$  Then,
\begin{enumerate}[resume]
        \item \label{i: push2} 
        if $f \colon S \to Y$ proper then there is a commutative diagram
\begin{equation*}
\begin{tikzcd}
 K_0^{\op{alg}}(S) \ar[d] \ar[r, "\alpha_*"]& K_0^{\op{top}}(S) 
     \ar[d, " i_{<0}^* h_{\op{top}}^{-1}f_*"] \\
            K_0(\op{MF}(Y, w))  \ar[r, "\alpha_{(Y, w)} "] & \cc K^0(Y, w),
\end{tikzcd}
\end{equation*}
where the left map is induced by the pushforward $\op{D}^b(S) \to \op{D}^b(X);$
\item \label{i: pull2}
there is a commutative diagram
\begin{equation*}
\begin{tikzcd}
      K_0(\op{MF}(Y, w)) \ar[d, "f^*"] \ar[r, "\alpha_{(Y, w)} "] & \cc K^0(Y, w) \ar[d, "f^*"] \\
 K_0(\op{MF}(S, 0))  \ar[r, "\alpha^*_{\ZZ/2}"]& K^0_{\op{top}}(S). 
\end{tikzcd}
\end{equation*}
\end{enumerate}
Furthermore, the map $\alpha_{(Y, w)}$ satisfies the following compatibilities:
\begin{enumerate}[resume]
        \item \label{i:module}
the following diagram commutes
        \[  \begin{tikzcd}
         K^0_{\op{alg}}(X) \otimes K_0(\op{MF}(Y, w)) \ar[d, "\cap"] \ar[r, "\alpha^* \otimes \alpha_{(Y, w)}"] & K^0_{\op{top}}(X) \otimes \cc K^0(Y, w) \ar[d, "\cap"] \\
   K_0(\op{MF}(Y, w))  \ar[r, "\alpha_{(Y, w)}"] & \cc K^0(Y, w);
   \end{tikzcd}
        \] 
        \item \label{i: shift} we have
        $$\alpha_{(Y, w)}(\cc E[1]) = -\alpha_{(Y, w)}(\cc E).$$
        \end{enumerate}
\end{theorem}

\begin{proof}

    To prove \eqref{i:push}, we use the $K$-theoretic Grothendieck--Riemann--Roch statement for the map $f \colon X \to X'$ of singular varieties \eqref{e:GRRKSing}.
Combining this with \eqref{e:PushZCrit}, we have 
 \[  \begin{tikzcd}
     K_0(\op{D}^b(X)) \ar[d, "f_*"] \ar[r, "i_{<0}^* h_{\op{top}}^{-1} \alpha_*"]& \cc K^0(Y, w)
     \ar[d, "f_*"] \\
            K_0(\op{D}^b(X'))  \ar[r, "(i ')_{<0}^* h_{\op{top}}^{-1} \alpha_* ' "] & \cc K^0(Y ', w'),
        \end{tikzcd}
        \] 
which implies the statement.

        Property \eqref{i:pull} follows from the description of $\alpha_{(Y, w)}$ in Theorem~\ref{t:MFch}.  
        For a matrix factorization $\cc E = (E_0, E_1, \phi_0, \phi_1) \in \op{Ob}\left( \op{MF}(Y, w)\right)$, the pullback is given by $f^* \cc E = (f^*E_0, f^*E_1, f^*\phi_0, f^*\phi_1).$  We have
        \begin{align*}
           f^*\left(  \alpha_{(Y', w')}(\cc E)
            \right)
           &= f^*\left(i_{<0}^*\left(\alpha^*(E_\bullet)\right)  \right) \\
           &= i_{<0}^*\left(f^*(\alpha^* (E_\bullet)) \right) \\
           &= i_{<0}^*\left(\alpha^*(f^* E_\bullet) \right) \\
           &= \alpha_{(Y, w)}(f^* \cc E).
        \end{align*} 

         Property \eqref{i: push2} is due to the commutative diagram
        \begin{equation}
\begin{tikzcd}
K_0(\op{D}^b(S)) \ar[d, "f_*"] \ar[r, "\alpha_*"]& K_0^{\op{top}}(S) \ar[d, "f_*"] \\
 K_0(\op{D}^b(X)) \ar[d] \ar[r, "\alpha_*"]& K_0^{\op{top}}(X) 
     \ar[d, " i_{<0}^* h_{\op{top}}^{-1}"] \\
            K_0(\op{D}_{\op{Sg}}(X))  \ar[r, "\alpha_{\op{Sg}} "] & \cc K^0(Y, w).
\end{tikzcd}
\end{equation}
The top square commutes because of the functoriality of $\alpha_*$, while the bottom square commutes simply by the definition of $\alpha_{\op{Sg}}$.  

Property~\eqref{i: pull2} follows from the fact that $\alpha^*$ commutes with pullback.  We have
     \begin{align*}
        f^*\left(\alpha_{(Y, w)}( \cc E)\right) 
            &=
            f^*\left(\alpha^*[E_0 \xrightarrow{\phi_0} E_1]\right) \\
           &= [\alpha^*(f^*( E_0))] - [\alpha^*(f^*( E_1))] \\
           &= \alpha^*_{\ZZ/2}\left( f^* (\cc E)
            \right).
        \end{align*}

        Property \eqref{i:module} follows from \eqref{e:module}.  Property \eqref{i: shift} is due to the fact that the equivalence of categories $\Sigma$ respects the triangulated structure, and \newline $\alpha_*\colon K_0^{\op{alg}}(X) \to K_0^{\op{top}}(X)$ is a group homomorphism.

\end{proof}
Item \eqref{i: push2} may be viewed as a special case of \eqref{i:push} for the
map of LG models $f \colon (S, 0) \to (Y, w)$, except that here we do not require that $S$ be smooth.

As an application of \eqref{i: push2}, we have the following important compatibility for geometric phases, which is closely related to Kn\"orrer periodicity (see the next section for a similar statement).  
\begin{corollary}
    Let $Z \subset X$ and $(Y, w)$ be as in Example~\ref{e:geomphase}.  Then the following diagram commutes:
     \[  \begin{tikzcd}
         K^0_{\op{alg}}(Z) \ar[d, "j'_* \circ \pi|_Z^*"] \ar[r, "\alpha^*"] &  \cc K^0_{\op{top}}(Z) \ar[d, " i_{<0}^* h_{\op{top}}^{-1} j'_* \circ \pi|_Z^*"] \\
     K_0(\op{MF}_{\CC^*}(Y, w))  \ar[r, "\alpha_{(Y, w)}"]& \cc K^0(Y, w).
        \end{tikzcd}
        \] 
\end{corollary}
By Proposition~\ref{p:BPK} and Theorem~\ref{t:ish}, the vertical maps are isomorphisms, and so
the above corollary may be interpreted as saying that if we represent the complete intersection $Z$ by the LG model $(Y, w)$, then map $\alpha_{(Y, w)}$ is identified with the natural map $\alpha^*$.

\subsection{Grothendieck--Riemann--Roch in cohomology}

Theorem~\ref{t:GRR} and the covariance of $\ch_*$ immediately imply the following.
\begin{corollary}\label{c:GRR}
If $f \colon Y \to Y'$ is a proper map of LG models as in the previous theorem, then there is a commutative diagram
     \[  \begin{tikzcd}
     K_0(\op{MF}(Y, w)) \ar[d, "f_*"] \ar[r, "\tau_{(Y, w)}"]& \cc H(Y, w; \QQ) 
     \ar[d, "f_*"] \\
            K_0(\op{MF}(Y', w'))  \ar[r, "\tau_{(Y ', w ')}"] & \cc H (Y', w'; \QQ).
        \end{tikzcd}
        \] 
\end{corollary}

     \begin{remark}\label{r:cohom}
         Properties~\eqref{i:pull}-\eqref{i: shift} of Theorem~\ref{t:GRR} also have natural counterparts in cohomology.  The statements are essentially the same, with $\tau_{(Y, w)}$ taking the place of $\alpha_{(Y, w)}.$  The proofs follow immediately from Theorem~\ref{t:GRR} and the discussion in Section~\ref{ss:cc}.
     \end{remark}   

\section{Sebastiani--Thom sums}
In this section we prove the compatibility of the map $\alpha_{(Y, w)}$ with Sebastiani--Thom sums.  The Sebastiani--Thom sum of two LG models may be viewed as the analog of the product of varieties.  We prove that the map $\alpha_{(Y, w)}$ is compatible with the associated ``K\"unneth map'' on critical $K$-theory.

Let  $w_1\colon Y_1 \to \CC $ and $w_2\colon Y_2 \to \CC$  be two LG models.  Let $Z_1 = \op{Crit}(w_1)$, $Z_2 = \op{Crit}(w_2)$, $X_1 = w_1^{-1}(0)$, and $ X_2 = w_2^{-1}(0)$. 
Let $Y$ denote the product $Y_1 \times Y_2$ and define the \emph{Sebastiani--Thom sum} to be $$w = w_1 \boxplus w_2\colon Y \to \CC$$ where $w_1 \boxplus w_2(y_1, y_1) = w_1(y_1) + w_2(y_2).$  Let $Z = \op{Crit}(w)$  and $X = w^{-1}(0)$.
Recall the exterior tensor product described in Section~\ref{ss:func}
\begin{align}
    K_0(\op{MF}(Y_1, w_1)) \otimes K_0(\op{MF}(Y_2, w_2)) &\to K_0(\op{MF}(Y, w)) \\
    \cc E\otimes \cc F &\mapsto \cc E \boxtimes \cc F:= \pi_1^*(\cc E) \otimes \pi_2^*(\cc F).   \nonumber 
\end{align}
There is an analogous map in critical $K$-theory.  
\begin{proposition}
    There is a well-defined external product
    $$\cc K^0(Y_1, w_1) \otimes \cc K^0(Y_2, w_2) \to \cc K^0(Y, w).$$
\end{proposition}

\begin{proof}
    For this statement it is necessary to use the alternative description of $\cc K^0(Y, w)$ as the direct limit $\varinjlim_V K^0_{\op{top}}(V, V_{<0})$
    over all neighborhoods $V$ of the critical locus $Z$ (Proposition~\ref{p:ZvX}).
    The key point is that  $Z$ is equal to  $Z_1 \times Z_2$. Thus, given neighborhoods $U_1, U_2$ of $Z_1$, $Z_2$ respectively, we have a neighborhood $U = U_1 \times U_2 $ of $Z$.  
      Let $U_{<0} = U \cap Re(w)^{-1}(-\infty, 0)$.  Note that if $Re(w) <0$, then $Re(w_1)<0$ or $Re(w_2)<0$ (or both), so $U_{<0} \subset (U_1)_{<0} \times U_2 \cup  U_1 \times (U_2)_{<0}$.
   We can have a  pullback map
   $    K^0(U,  (U_1)_{<0} \times U_2 \cup  U_1 \times (U_2)_{<0}) \to K^0(U, U_{<0}).$
    
    We have the following:
    \begin{equation}\label{e:externalProdTop}  
    \begin{tikzcd}
     K^0_{\op{top}}(U_1, (U_1)_{<0}) \otimes K^0_{\op{top}}(U_2, (U_2)_{<0}) \ar[d, " \pi_1^* \otimes \pi_2^*"] \\
    K^0_{\op{top}}(U_1 \times U_2, (U_1)_{<0} \times U_2) \otimes K^0_{\op{top}}(U_1 \times U_2, U_1 \times (U_2)_{<0}) \ar[d, "- \otimes - "] \\
    K^0_{\op{top}}(U,  (U_1)_{<0} \times U_2 \cup  U_1 \times (U_2)_{<0}) \ar[d]\\
    K^0_{\op{top}}(U, U_{<0}).
    \end{tikzcd}
    \end{equation}
    The external product is then induced by this composition.  
    \end{proof}

As a special case of  \eqref{e:externalProdTop}, one has the external product $$K^0_{\op{top}}(Y_1, (Y_1)_{<0}) \otimes K^0_{\op{top}}(Y_2, (Y_2)_{<0}) \to K^0_{\op{top}}(Y, Y_{<0}),$$ and the following diagram commutes
\begin{equation}\label{e:commtens1}
    \begin{tikzcd}K^0_{\op{top}}(Y_1, (Y_1)_{<0}) \otimes K^0_{\op{top}}(Y_2, (Y_2)_{<0}) \ar[d] \ar[r] & K^0_{\op{top}}(Y, Y_{<0}) \ar[d] \\
    \cc K^0(Y_1, w_1) \otimes \cc K^0(Y_2, w_2) \ar[r] & \cc K^0(Y, w).
    \end{tikzcd}
    \end{equation}
where the vertical maps are the respective pullbacks $i_{<0}^*$.

Recall from  Section~\ref{s:checha}  the map
\begin{align*}
\op{Ob}\left( \op{MF}(Y, w)\right) & \to K^0_{\op{top}}(Y, Y \setminus X) \\
\cc E &\mapsto \alpha^*(E_\bullet),
\end{align*}
where,
for $\cc E \in \op{MF}(Y, w)$,
$E_\bullet$ denotes the two-term complex $E_1 \xrightarrow{\phi_1} E_0$.

\begin{proposition}\label{p:tensor1}
Given $\cc E \in \op{Ob}\left( \op{MF}(Y_1, w_1)\right)$ and $\cc F \in \op{Ob}\left( \op{MF}(Y_2, w_2)\right)$, let $\cc T = \cc E \boxtimes \cc F \in \op{Ob}\left( \op{MF}(Y, w)\right)$.  Then
$$\alpha^*(E_\bullet) \boxtimes \alpha^*(F_\bullet) = \alpha^* (T_\bullet)$$
in $K^0_{\op{top}}(Y, Y_{<0}).$

\end{proposition}

\begin{proof}
Given  $\cc E $ and $\cc F$ as in the proposition, 
we have $\alpha^* E_\bullet = E_1 \xrightarrow{\phi^E_1} E_0 \in K^0_{\op{top}}(Y_1, (Y_1)_{<0})$, $\alpha^* F_\bullet = F_1 \xrightarrow{\phi^F_1} F_0 \in K^0_{\op{top}}(Y_2, (Y_2)_{<0})$ and the external tensor product is the three-term complex $S_\bullet \in K^0_{\op{top}}(Y, Y_{<0})$ given by
\[
\begin{tikzcd}[ampersand replacement=\&]
E_1 \boxtimes F_1 
\arrow[r, swap, "{{\begin{pmatrix} 
\phi^F_1 \\ -\phi^E_1 
\end{pmatrix}}}"] \&  E_1 \boxtimes F_0 \oplus E_0 \boxtimes F_1  \ar[r, swap, "{{\begin{pmatrix} 
\phi^E_1
 & 
 \phi^F_1  
\end{pmatrix}}}"]
\& E_0 \boxtimes F_0
\end{tikzcd}
\]
where to simplify notation we also denote by $E_i/F_i$ the associated vector bundle on $Y$ obtained by pullback, and suppress the identity maps (e.g. the map $id_{E_1} \boxtimes \phi_1^F \colon \pi_1^* E_1 \otimes \pi_2^* F_1 \to \pi_1^* E_1 \otimes \pi_1^* F_0$ is denoted simply as $\phi_1^F\colon E_1 \boxtimes F_1 \to E_1 \boxtimes F_0$).

If one first takes the tensor product of matrix factorizations $\cc T = \cc E \boxtimes \cc F \in \op{Ob}\left( \op{MF}(Y, w)\right)$, then the associated two-term complex $T_\bullet \in K^0_{\op{top}}(Y, Y_{<0})$ is given by 
\[
\begin{tikzcd}[ampersand replacement=\&]
E_1 \boxtimes F_0 \oplus E_0 \boxtimes F_1
\arrow[r, swap, "{{\begin{pmatrix} 
\phi^F_0 &  \phi^E_0 \\
-\phi^E_1 & \phi^F_1 
\end{pmatrix}}}"] \&  
E_1 \boxtimes F_1 \oplus E_0 \boxtimes F_0 
\end{tikzcd}.
\]
We must show that $S_\bullet$ and $T_\bullet$ represent the same class in $K^0_{\op{top}}(Y, Y_{<0})$.  For this, consider $$
S_\bullet \oplus T_\bullet[-1] = Q_{-2} \xrightarrow{\phi_{-2}} Q_{-1} \xrightarrow{\phi_{-1}} Q_0,
$$
with 
$$
\phi_{-2} = \left[ \begin{array}{ccc} 
\phi^F_1 & 0 & 0 \\
- \phi^E_1 & 0 & 0 \\
0& -\phi^F_0 & - \phi^E_0 \\
0 & \phi^E_1 & -\phi^F_1 
\end{array}\right], \hspace{1 cm} \phi_{-1} = \left[ \begin{array}{cccc} 
 \phi^E_1 
 & 
 \phi^F_1  &  0 & 0  
\end{array}\right].
$$
We will show that this class is zero in $K$-theory.

Define a complex $C_\bullet= Q_{-2} \xrightarrow{\psi_{-2}} Q_{-1} \xrightarrow{\psi_{-1}} Q_0$
with the same terms as above but with maps given by 
$$
\psi_{-2} = \left[ \begin{array}{ccc} 
0 & -id & 0 \\
0 & 0 & id \\
id& 0 & 0 \\
0 & 0 & 0 
\end{array}\right], \hspace{1 cm} \psi_{-1} = \left[ \begin{array}{cccc} 
 0
 & 
 0  &  0 & id  
\end{array}\right].
$$
This complex is exact on $Y$.

Define the complex $C_\bullet(t)$ to again have the same terms, but with $\psi_{i}(t) = t\psi_{i} + (1-t)\phi_{i}$ for $i = -2, -1.$  Then $C_\bullet(1) = C_\bullet$ and $C_\bullet(0) = S_\bullet \oplus T_\bullet[-1]$. Note that $\psi_{-1}(t) \circ \psi_{-2}(t) = 0.$

We claim that for $0 \leq t \leq 1$, $C_\bullet(t)$ is exact on $Y_{<0}$.  This is clear for $t=  0, 1$, so we will assume $0<t<1$ in what follows.  
The map $\psi_{-1}(t)$ is always surjective, so it suffices to verify that $\psi_{-2}(t)$ is injective.  
On a sufficiently small open subset of $Y$, these maps may be represented as matrices.  
Let $C_1, C_2, C_3$ denote the column blocks of $\psi_{-2}(t)$.  That is, 
$\psi_{-2}(t) = \left[ C_1 | C_2 | C_3\right]$, with 
$$ C_1 = \left[ \begin{array}{c} 
(1-t) \phi^F_1  \\
-(1-t) \phi^E_1  \\
t\cdot id \\
0  
\end{array}\right], \;\; C_2 = \left[ \begin{array}{c} 
-t \cdot id  \\
0  \\
-(1-t) \phi^F_0 \\
(1-t) \phi_1^E
\end{array}\right], \;\; C_3 = \left[ \begin{array}{c} 
0  \\
t \cdot id  \\ 
-(1-t) \phi^E_0 \\
-(1-t) \phi_1^F
\end{array}\right].$$
Apply the column operations
$$C_1 \mapsto \tilde C_1 = C_1 + \frac{1-t}{t} C_2 \circ(id_{E_1} \boxtimes \phi_1^F) + \frac{1-t}{t} \circ C_3 (\phi_1^E \boxtimes id_{F_1}).$$
The new matrix is then given by $\left[ \tilde C_1 | C_2 | C_3\right]$ with 
$$ \tilde C_1 =\left[ \begin{array}{c} 
0  \\
0  \\
t\cdot id - \frac{(1-t)^2}{t} w_2 - \frac{(1-t)^2}{t} w_1 \\
0  
\end{array}\right] = \left( t - \frac{(1-t)^2}{t}(w_1+w_2)\right)\left[ \begin{array}{c} 
0  \\
0  \\
 id \\
0  
\end{array}\right].$$  The scalar $\left( t - \frac{(1-t)^2}{t}(w_1+w_2)\right)$ is zero only if 
$$ t^2 = (1-t)^2(w_1 + w_2),$$
which never occurs on $Y_{<0}$ for $0<t<1$, since the left hand side is positive and the right hand side has negative real part.  From this it is easily seen that the map $\left[ \tilde C_1 | C_2 | C_3\right]$ is full rank on $Y_{<0}$ and thus so is $\psi_{-2}(t)$.

We have shown that  $S_\bullet \oplus T_\bullet[-1]$ and $C_\bullet$ are related by a one-parameter family of complexes exact on $Y_{<0}$.  By condition (3) of Definition~\ref{d:pair}, 
they represent the same class in $K^0_{\op{top}}(Y, Y_{<0})$.  Because $C_\bullet$ is exact on all of $Y$ it is zero in $K^0_{\op{top}}(Y, Y_{<0})$, which implies the claim.
\end{proof}

Proposition~\ref{p:tensor1} and the commuting square \eqref{e:commtens1}  imply the following compatibility.
\begin{theorem}\label{t:ST}
The external product is compatible with the map $\alpha_{(Y, w)}$.  That is, the following diagram commutes:
\begin{equation}\label{e:st}
\begin{tikzcd}
K_0(\op{MF}(Y_1, w_1)) \otimes K_0(\op{MF}(Y_2, w_2))  \ar[d, "-\boxtimes-"] \ar[rr, "\alpha_{(Y_1, w_1)} \otimes \alpha_{(Y_2, w_2)}"] && 
\cc K^0(Y_1, w_1) \otimes \cc K^0(Y_2, w_2) \ar[d, "-\boxtimes-"]\\
 K_0(\op{MF}(Y, w)) \ar[rr, "\alpha_{(Y, w)}"] && \cc K^0(Y, w).
\end{tikzcd}
\end{equation}
\end{theorem}

As a particular example, consider the LG model given by the function $x^2 + y^2$ on $\CC^2.$ We have
$K_0(\op{MF}(\CC^2, x^2 + y^2)) = \ZZ$ by Example~\ref{e:geomphase}, with a generator given by 
$$\cc G = (\cc O_{\CC^2}, \cc O_{\CC^2}, x+iy, x-iy).$$
In critical $K$-theory, by Proposition~\ref{p:BPK}, $\cc K^0(\CC^2, x^2 + y^2) = K_{\op{top}}(pt) \cong \ZZ$.
The matrix factorization $\cc G$ is mapped to a generator $\beta$ via 
$$K_0(\op{MF}(\CC^2, x^2 + y^2)) \xrightarrow{\alpha_{(\CC^2, x^2 + y^2)}} \cc K^0(\CC^2, x^2 + y^2).$$
The following corollary generalizes part of \cite[Theorem~3.33]{Brown} from quasihomogeneous hypersurface singularities to arbitrary LG models, and may be viewed as a compatibility between Kn\"orrer periodicity and Bott periodicity.
\begin{corollary} \label{c:KvB} The following diagram commutes:
    \begin{equation}\label{e:st1}
\begin{tikzcd}
K_0(\op{MF}(Y, w))   \ar[d, "- \boxtimes \cc G" ] \ar[rr, "\alpha_{(Y, w)} "] & &
\cc K^0(Y, w)  \ar[d, "- \boxtimes \beta  "]\\
 K_0(\op{MF}(Y \times \CC^2, w + x^2 + y^2)) \ar[rr, "\alpha_{(Y\times \CC^2, w + x^2 + y^2)}"] & &\cc K^0(Y\times \CC^2, w + x^2 + y^2).
\end{tikzcd}
\end{equation}
   
\end{corollary}

As a corollary to Theorem~\ref{t:ST}, we obtain an analogous result for the internal tensor product.
Let $v, w \colon Y \to \CC$ be two regular functions on $Y$. The map 
$$\otimes_{\cc O_Y}: K_0(\op{MF}(Y, v)) \otimes K_0(\op{MF}(Y, w)) \to K_0(\op{MF}(Y, v + w)) $$
is equal to the composition 
$$K_0(\op{MF}(Y, v)) \otimes K_0(\op{MF}(Y, w)) \to K_0(\op{MF}(Y\times Y, v \boxplus w)) \xrightarrow{\Delta^*} K_0(\op{MF}(Y, v+w)),$$
where $$\Delta: (Y, v+w) \to (Y \times Y, v \boxplus w)$$
is the diagonal morphism of LG models. 

Similarly, in $K$-theory, define 
$\cc K^0_{\op{top}}(Y, v) \otimes \cc K^0_{\op{top}}(Y, w) \to \cc K^0_{\op{top}}(Y, v+w)$ to be the composition
$$\cc K^0_{\op{top}}(Y, v) \otimes \cc K^0_{\op{top}}(Y, w) \to \cc K^0_{\op{top}}(Y\times Y, v \boxplus w) \xrightarrow{\Delta^*} \cc K^0_{\op{top}}(Y, v+w).$$
Then by Theorem~\ref{t:ST} and  Theorem~\ref{t:GRR}~\eqref{i:pull}, 
we obtain.
\begin{corollary}
    The following diagram commutes:
\begin{equation}\label{e:int}
\begin{tikzcd}
K_0(\op{MF}(Y, v)) \otimes K_0(\op{MF}(Y, w))  \ar[d] \ar[rr, "\alpha_{(Y, v)} \otimes \alpha_{(Y, w)}"] && 
\cc K^0(Y, v) \otimes \cc K^0(Y, w) \ar[d]\\
 K_0(\op{MF}(Y, v+w)) \ar[rr, "\alpha_{(Y, v+w)}"] && \cc K^0(Y, v+w).
\end{tikzcd}
\end{equation}
\end{corollary}

\bibliographystyle{alpha}
\bibliography{references2}

\end{document}